\documentclass[10pt,twoside]{amsart}
\usepackage{mathtools,amsmath,amssymb}
\usepackage{mathrsfs}

\usepackage[a4paper,top=1in, bottom=1in, left=1in, right=1in]{geometry} % resize page
\usepackage{graphicx} % Required for inserting images
\usepackage[colorlinks=true,linkcolor=blue,citecolor=black]{hyperref} % Enable hyperlinks
\usepackage{cleveref}
\usepackage[backend=biber,backref=true,url=false,maxnames=9,isbn=false,sorting=nyt,style=numeric-comp]{biblatex} % Enable backref
\renewbibmacro*{pageref}{%
  \iflistundef{pageref}
    {}
    {\addspace% Custom separator
     \mkbibbrackets{\printlist[pageref][-\value{listtotal}]{pageref}}}} 
\usepackage{lipsum}% for rescaling tikzcd 
\usepackage{adjustbox} % for rescaling tikzcd
\usepackage{tikz-cd}
\usepackage{mathptmx}
\usepackage{quiver}
\usepackage{float}
\usepackage{amssymb}
\usepackage[english]{babel}
\usepackage{amsthm}
\usepackage{amsmath}
\usepackage{mathtools}
\usepackage{soul}
\usepackage{nicematrix}
\usepackage{array}   
\usepackage{booktabs}
\usepackage{float}
\usepackage{orcidlink}
\usepackage[toc,page]{appendix}
\usepackage{caption}
\usepackage{mathrsfs}     
\makeatletter
\renewcommand*\env@matrix[1][*\c@MaxMatrixCols c]{%
  \hskip -\arraycolsep
  \let\@ifnextchar\new@ifnextchar
  \array{#1}}
\makeatother

\usepackage{makecell, caption, chngcntr} %

\usepackage{siunitx} %
\counterwithin{table}{section}

\newcommand{\PSL}{\operatorname{PSL}}
\newcommand{\fq}{\mathbb{F}}
\newcommand{\SL}{\operatorname{SL}}
\newcommand{\Q}{\mathbb Q}
\newcommand{\Gal}{\mathrm{Gal}}
\newcommand{\e}{\varepsilon}
\newcommand{\Z}{\mathbb Z}
\newcommand{\irr}{\mathrm{Irr}}

\newcommand{\M}{\mathrm{M}}
\newcommand{\cd}{\mathrm{cd}}

\usepackage{thmtools, thm-restate}
\declaretheorem{theorem}
\newtheorem{lemma}{Lemma}[section]

\theoremstyle{definition} % Sets non-italic, upright font

\title[Wedderburn Decomposition of the Rational Group Algebras of $\SL_2(q)$ and $\PSL_2(q)$]{Wedderburn decomposition of the rational group algebras of $\SL_2(q)$ and $\PSL_2(q)$}
\date{\today}
\author[Choudhary]{Ram Karan Choudhary\orcidlink{0009-0003-7688-6397}}
\email[(Choudhary)]{ramkchoudhary1997@gmail.com, ram.choudhary@iiserpune.ac.in}
\address{Indian Institute of Science Education and Research Pune, Dr. Homi Bhabha Road, Pashan, Pune 411008, India}
\author[Panja]{Saikat Panja\orcidlink{0000-0002-9639-3122}}
\email[(Panja)]{panjasaikat300@gmail.com}
\address{Indian Statistical Institute, Bengaluru Centre, 8th Mile, Mysore Rd, RVCE Post, Gnana Bharathi, Bengaluru, Karnataka 560059, India}
\thanks{Choudhary is supported by a postdoctoral fellowship from IISER Pune. Panja is supported by an NBHM postdoctoral fellowship file number ending at R\&D-II/6746.}
\date{\today}
\dedicatory{}
\subjclass[2020]{primary 20C05; secondary 20G15, 20G05}
\keywords{rational group algebra; Wedderburn decomposition; finite groups of Lie type}
\begin{document}
\begin{abstract}
In this article, we derive explicit combinatorial formulas, depending only on $q$, for the Wedderburn decomposition of the rational group algebras of the finite linear groups $\operatorname{SL}_2(q)$ and $\operatorname{PSL}_2(q)$. Furthermore, we also determine the number of pairwise non-isomorphic simple $\Q G$-modules of each possible dimension for $G$ being either $\operatorname{SL}_2(q)$ or $\operatorname{PSL}_2(q)$.
\end{abstract}

\maketitle
\tableofcontents
\section{Introduction}\label{sec:intro}
Let $G$ be a group, and let $\mathbb{F}$ be a number field. Then the \emph{group algebra} of $G$ over $\mathbb{F}$ is
\[
\mathbb{F}G = \{\, f: G \to \mathbb{F} \, : \, f(g) = 0 \text{ for all but finitely many } g \in G \,\},
\]
where addition, scalar multiplication, and multiplication are defined by
\[
(f+h)(g) = f(g) + h(g), \,\,\,\, (af)(g) = a\,f(g), \,\,\,\, \text{and} \,\,\,\, (fh)(g) = \sum_{xy=g} f(x)h(y)
\]
for $f,h \in \mathbb{F}G$, $a \in \mathbb{F}$, and $g \in G$, respectively. Note that $\mathbb{F}G$ is a free $\mathbb{F}$-module with basis $G$. Describing group algebras is a classical problem in algebra, allowing many approaches. When considering semisimple algebras, i.e., the case of a finite group over a field whose characteristic does not divide the group order, this involves describing the Wedderburn components of the group algebra. By the Wedderburn--Artin theorem, a semisimple group algebra $\mathbb{F}G$ decomposes as a direct sum of matrix algebras over division rings
\[
\mathbb{F}G \cong \bigoplus_{i=1}^{r} \M_{n_i}(D_i),
\]
where each $\M_{n_i}(D_i)$ is a full matrix algebra of size $n_i$ over a division ring $D_i$, {\color{black}which is} finite-dimensional over its center. These are called the \emph{simple components} of $\mathbb{F}G$. Moreover, by the Brauer--Witt theorem \cite{Yam}, each simple component is Brauer equivalent to a cyclotomic algebra. The study of the Wedderburn decomposition of rational group algebras has received significant attention due to its importance in understanding various algebraic structures (see \cite{Herman, Jes-Rio, Rit-Seh}).  

This decomposition can be achieved explicitly using character-theoretic methods or by exploiting subgroup structures via Shoda pairs for certain classes of rational group algebras, in particular for rational group algebras of monomial groups. See the works on rational group algebras based on Shoda pair theory presented in~\cite{Shoda1, Shoda2, Shoda3, Shoda4, Shoda5, Shoda6}, which have been done in recent years. Computational approaches, such as the \textsc{Wedderga} package in {\sf GAP}~\cite{Gap}, implement these methods to compute the Wedderburn decomposition of rational group algebras of such groups; however, exact computations remain challenging, especially for groups of large order. 

For finite abelian groups, Perlis and Walker \cite{PW} provided a combinatorial formula for the Wedderburn decomposition of their rational group algebras based on counting cyclic subgroups. For several classes of finite $p$-groups, the first author and Prajapati have developed combinatorial techniques to determine the Wedderburn decomposition explicitly, either by reducing to abelian sections or using parameters from group presentations, as detailed in \cite{Ram1, Ram2, Ram3, Ram4, Ram5, Ram6}.  

Let $\mathbb{F}_q$ denote the finite field of order $q$, where $q=p^m$ for some prime $p$ and positive integer $m$. Let $\SL_2(q)$ and $\PSL_2(q)$ denote the groups of $2 \times 2$ matrices over $\mathbb{F}_q$ with determinant $1$ and their projective version, respectively, i.e.,
\[
\SL_2(q) = \{ A \in \mathrm{GL}_2(\mathbb{F}_q) \, : \, \det(A) = 1 \}\,\,\,\, \text{and} \,\,\,\, 
\PSL_2(q) = \SL_2(q)/Z(\SL_2(q)),
\]
where $Z(\SL_2(q))$ is the center of $\SL_2(q)$. In this article, we derive combinatorial formulas for the Wedderburn decomposition of the rational group algebras of the finite classical matrix groups $\SL_2(q)$ and $\PSL_2(q)$.

For a positive integer $d$, let $\zeta_d$ denote a primitive $d$-th root of unity in $\mathbb{C}$. Note that $\PSL_2(q)=\SL_2(q)$ for $q=2^m$. Theorem~\ref{thm:wedderburn PSL(even)} determines a formula, depending only on $q$, for the Wedderburn decomposition of the rational group algebras of $\PSL_2(q)=\SL_2(q)$, where $q$ is a power of $2$.

\begin{theorem}\label{thm:wedderburn PSL(even)}
	Let $G = \PSL_2(q) = \SL_2(q)$, where $q$ is a power of $2$. Suppose that 
\[
A = \{ \, 2 < d \leq q+1 \, : \, d \mid q+1 \, \} \,\,\,\, \text{and} \,\,\,\, 
B = \{ \, 2 < d \leq q-1 \, : \, d \mid q-1 \, \}.
\] 
Then the Wedderburn decomposition of $\mathbb{Q}G$ is given by
	$$\mathbb{Q}G \cong \Q \bigoplus_{d \in A} \M_{q-1}(\Q(\zeta_{d}+\zeta_d^{-1})) \bigoplus \M_q(\Q) \bigoplus_{d \in B} \M_{q+1}(\Q(\zeta_{d}+\zeta_d^{-1})).$$
\end{theorem}

Next, we prove Theorem~\ref{thm:wedderburn PSL(odd)}, which provides a formula, depending only on $q$, for the Wedderburn decomposition of the rational group algebra of $\PSL_2(q)$, where $q$ is a power of an odd prime $p$.
\begin{theorem}\label{thm:wedderburn PSL(odd)}
	Let $G = \PSL_2(q)$, where $q$ is a power of an odd prime $p$. Suppose that
\[
A=\left\{\, 2 < d \leq \tfrac{q+1}{2} \, : \, d \mid \tfrac{q+1}{2} \,\right\}
\,\,\,\, \text{and} \,\,\,\,
B=\left\{\, 2 < d \leq \tfrac{q-1}{2} \, : \, d \mid \tfrac{q-1}{2} \,\right\}.
\]
Then we have the following.
    \begin{enumerate}
        \item {\bf Case ($q\equiv3\pmod{4}$).} In this case, the Wedderburn decomposition of $\mathbb{Q}G$ is given by
	$$\mathbb{Q}G \cong \Q \bigoplus \M_{\frac{q-1}{2}}(\Q(\sqrt{-q})) \bigoplus_{d \in A} \M_{q-1}(\Q(\zeta_{d}+\zeta_d^{-1})) \bigoplus \M_q(\Q) \bigoplus_{d \in B} \M_{q+1}(\Q(\zeta_{d}+\zeta_d^{-1})).$$

    \item {\bf Case ($q\equiv1\pmod{4}$).} In this case, we have the following two subcases.
    \begin{enumerate}
        \item {\bf Subcase ($q$ is not a square)}. In this subcase, the Wedderburn decomposition of $\mathbb{Q}G$ is given by
        $$\mathbb{Q}G \cong \Q \bigoplus \M_{\frac{q+1}{2}}(\Q(\sqrt{q})) \bigoplus_{d \in A} \M_{q-1}(\Q(\zeta_{d}+\zeta_d^{-1})) \bigoplus \M_q(\Q) \bigoplus_{d \in B} \M_{q+1}(\Q(\zeta_{d}+\zeta_d^{-1})).$$

        \item {\bf Subcase ($q$ is a square)}. In this subcase, the Wedderburn decomposition of $\mathbb{Q}G$ is given by
        $$\mathbb{Q}G \cong \Q \bigoplus 2\M_{\frac{q+1}{2}}(\Q) \bigoplus_{d \in A} \M_{q-1}(\Q(\zeta_{d}+\zeta_d^{-1})) \bigoplus \M_q(\Q) \bigoplus_{d \in B} \M_{q+1}(\Q(\zeta_{d}+\zeta_d^{-1})).$$
    \end{enumerate}
    \end{enumerate}
\end{theorem}

A quaternion algebra over a field $\mathbb{F}$ with $\operatorname{char}(\mathbb{F}) \neq 2$ is a central simple algebra of dimension $4$, denoted $(a,b)_{\mathbb{F}}$, generated by elements $i, j$ satisfying
\[
i^2 = a, \quad j^2 = b, \quad ij = -ji, \quad a,b \in \mathbb{F}^\times.
\]

If $q$ is a power of $2$, then $\SL_2(q) = \PSL_2(q)$. Thus, with the above setup of division rings, we finally prove Theorem~\ref{thm:wedderburn SL(odd)}, which provides a formula depending only on $q$ for the Wedderburn decomposition of the rational group algebra of $\SL_2(q)$, where $q$ is a power of an odd prime $p$.
\begin{theorem}\label{thm:wedderburn SL(odd)}
	Let $G = \SL_2(q)$, where $q$ is a power of an odd prime $p$. Suppose that
\begin{align*}
        & A=\left\{\, 2 < d \leq \tfrac{q+1}{2} \, : \, d \mid \tfrac{q+1}{2} \, \right\}, \\
        & A'=\left\{\, 2 < d \leq q+1 \, : \, d \mid (q+1), \ d \nmid \tfrac{q+1}{2} \, \right\},\\
        & B=\left\{\, 2 < d \leq \tfrac{q-1}{2} \, : \, d \mid \tfrac{q-1}{2} \, \right\}, \, \, \, \,\text{and}\\
        & B'=\left\{\, 2 < d \leq q-1 \, : \, d \mid (q-1), \ d \nmid \tfrac{q-1}{2}  \, \right\}.
    \end{align*}
Furthermore, let us define the division algebras $D_d \cong (\zeta_{d}-\zeta_d^{-1}, -q)_{\Q(\zeta_{d}+\zeta_d^{-1})}$, $D'_d \cong  (\zeta_{d}-\zeta_d^{-1}, -1)_{\Q(\zeta_{d}+\zeta_d^{-1})}$, $D'\cong (-1,-1)_{\Q(\sqrt q)}$, and 
 $$D''\cong \begin{cases}
(-1,-p)_\Q &  \text{ for } \, \, \, \, q=p^{2m}\, \, \, \, \text{and} \, \, \, \, p\equiv 1\pmod 4, \\
 (-\ell,-p)_\Q & \text{ for } \, \, \, \, q=p^{2m}\, \, \, \, \text{and} \, \, \, \, p\equiv 3\pmod 4,
\end{cases}$$
where $\ell$ is a prime satisfying $\ell\equiv 3\pmod 4$ and the Legendre symbol $\left(\dfrac{\ell}{p}\right)$ takes value $-1$. Then we have the following.
    \begin{enumerate}
        \item {\bf Case ($q\equiv3\pmod{4}$).} In this case, the Wedderburn decomposition of $\mathbb{Q}G$ is given by
	\begin{align*}
	    \mathbb{Q}G \cong & \Q \bigoplus \M_q(\Q) \bigoplus \M_{\frac{q-1}{2}}(\Q(\sqrt{-q})) \bigoplus_{d \in A'} \M_{\frac{q-1}{2}}(D_d) \bigoplus_{d \in A} \M_{q-1}(\Q(\zeta_{d}+\zeta_d^{-1}))\\
        & \bigoplus \M_{\frac{q+1}{2}}(\Q(\sqrt{-q})) \bigoplus_{d \in B'} \M_{\frac{q+1}{2}}(D_d) \bigoplus_{d \in B} \M_{q+1}(\Q(\zeta_{d}+\zeta_d^{-1})).
	\end{align*}

    \item {\bf Case ($q\equiv1\pmod{4}$).} In this case, we have the following two subcases.
    \begin{enumerate}
        \item {\bf Subcase ($q$ is not a square)}. In this subcase, the Wedderburn decomposition of $\mathbb{Q}G$ is given by
        \begin{align*}
            \mathbb{Q}G \cong & \Q \bigoplus \M_q(\Q) \bigoplus \M_{\frac{q-1}{4}}(D') \bigoplus_{d \in A'} \M_{\frac{q-1}{2}}(D_d) \bigoplus_{d \in A} \M_{q-1}(\Q(\zeta_{d}+\zeta_d^{-1}))\\
            & \bigoplus \M_{\frac{q+1}{2}}(\Q(\sqrt{q})) \bigoplus_{d \in B'} \M_{\frac{q+1}{2}}(D_d) \bigoplus_{d \in B} \M_{q+1}(\Q(\zeta_{d}+\zeta_d^{-1})).
        \end{align*}

        \item {\bf Subcase ($q$ is a square)}. In this subcase, the Wedderburn decomposition of $\mathbb{Q}G$ is given by
        \begin{align*}
            \mathbb{Q}G \cong & \Q \bigoplus \M_q(\Q) \bigoplus 2\M_{\frac{q-1}{4}}(D'') \bigoplus_{d \in A'} \M_{\frac{q-1}{2}}(D_d) \bigoplus_{d \in A} \M_{q-1}(\Q(\zeta_{d}+\zeta_d^{-1}))\\
            & \bigoplus 2\M_{\frac{q+1}{2}}(\Q) \bigoplus_{d \in B'} \M_{\frac{q+1}{2}}(D_d) \bigoplus_{d \in B} \M_{q+1}(\Q(\zeta_{d}+\zeta_d^{-1})).
        \end{align*}
    \end{enumerate}
    \end{enumerate}
\end{theorem}

The organization of the article is as follows. In Section~\ref{sec:preliminaries}, we introduce the main notation and prerequisites that are ubiquitous throughout this article. In Section~\ref{sec:prelim}, we address the results concerning simple $\Q G$-modules for $G = \PSL_2(q)$, and present the proof of Theorems~\ref{thm:wedderburn PSL(even)} and~\ref{thm:wedderburn PSL(odd)}. Finally, in Section~\ref{sec:SLresult}, we address the results concerning simple $\Q G$-modules for $G = \SL_2(q)$, and present the proof of Theorem~\ref{thm:wedderburn SL(odd)}.

\section{Preliminaries}\label{sec:preliminaries}
	In this section, we introduce the main notation and prerequisites used throughout the paper. Unless stated otherwise, $p$ denotes a prime number, and $\zeta_m$ denotes a primitive $m$th root of unity in $\mathbb{C}$. For a finite group $G$, we denote by $\irr(G)$ the set of irreducible complex characters of $G$, and by $\cd(G)=\{\chi(1) : \chi\in \irr(G)\}$ the set of character degrees of $G$. For $\chi\in \irr(G)$, we write $m_{\mathbb{Q}}(\chi)$ for the Schur index of $\chi$ over $\mathbb{Q}$. For a character $\chi\in\irr(G)$, we define $\Q(\chi)$ to be the field extension of $\Q$ generated by the character values $\chi(g)$ for all $g\in G$.
    We further define $\Omega(\chi)=m_{\mathbb{Q}}(\chi)\sum_{\sigma\in \mathrm{Gal}(\mathbb{Q}(\chi)/\mathbb{Q})}\chi^\sigma.$
We denote by $\irr_{\mathbb{Q}}(G)$ the set of irreducible rational characters of $G$. 
All other notations used in this article are standard.

 \subsection{Schur theory}\label{sec:Schur}
    	Let $G$ be a finite group. Define an equivalence relation on $\irr(G)$ by Galois conjugacy over $\mathbb{Q}$. Two irreducible characters $\chi, \psi \in \irr(G)$ are called \emph{Galois conjugates} over $\mathbb{Q}$ if there exists $\sigma \in \Gal(\mathbb{Q}(\chi)/\mathbb{Q})$ such that $\chi^\sigma = \psi$, where $\chi^\sigma(g)=\sigma(\chi(g))$ for each $g \in G$.
	
	\begin{lemma}\cite[Lemma~9.17]{I}\label{SC}
		Let $G$ be a finite group and $\chi \in \irr(G)$. Denote by $E(\chi)$ the Galois conjugacy class of $\chi$ over $\mathbb{Q}$. Then
		\[
		|E(\chi)| = [\mathbb{Q}(\chi) : \mathbb{Q}].
		\]
	\end{lemma}
	
	Let $G$ be a finite group, and let $\chi \in \operatorname{Irr}(G)$. Then Schur theory implies that there exists an irreducible $\mathbb{Q}$-representation $\rho$ of $G$ that affords the character $\Omega(\chi)$. Moreover, we have
\[
\rho(1) = \Omega(\chi)(1) = m_{\mathbb{Q}}(\chi)\,[\mathbb{Q}(\chi) : \mathbb{Q}]\,\chi(1).
\]
It follows that distinct Galois conjugacy classes correspond to distinct irreducible rational representations of $G$. Reiner~\cite[Theorem~3]{Reiner} further describes the structure of the simple components in the Wedderburn decomposition of the rational group algebra $\mathbb{Q}G$ associated with these rational representations. By the direct sum of representations, we mean the standard construction given by the block-diagonal sum in matrix form.

	\begin{lemma}\cite[Theorem~3]{Reiner} \label{Reiner}
		Let $\mathbb{K}$ be a field of characteristic zero and let $\overline{\mathbb K}$ denote its algebraic closure. Suppose $T$ is an irreducible $\mathbb{K}$-representation of $G$, extended linearly to a $\mathbb{K}$-representation of $\mathbb{K}G$. Set
		\[
		A = \{ T(x) : x \in \mathbb{K}G \}.
		\]
		Then $A$ is a simple $\mathbb{K}$-algebra and hence is isomorphic to $M_n(D)$ for some positive integer $n$ and some division algebra $D$ over $\mathbb{K}$.
		
    		Moreover, after extension of scalars to $\overline{\mathbb K}$, the representation $T$ decomposes as a direct sum
		\[
		T\otimes_{\mathbb{K}}\overline{\mathbb K}
		\;\cong\;
		m_{\mathbb{K}}(\chi)\bigoplus_{\sigma\in\Gal(\mathbb{K}(\chi)/\mathbb{K})} U^\sigma,
		\]
		where $U$ is an irreducible $\overline{\mathbb K}$-representation of $G$ with character $\chi$, 
		$U^\sigma$ affords the character $\chi^\sigma$, and 
		$m_{\mathbb{K}}(\chi)$ denotes the Schur index of $\chi$ over $\mathbb{K}$.  
		In this case, the center of $D$ satisfies
		\[
		Z(D)\cong \mathbb{K}(\chi),
		\]
		and the degree of $D$ over its center is given by
		\[
		[D:Z(D)]=\bigl(m_{\mathbb{K}}(\chi)\bigr)^2.
		\]
		
	\end{lemma}
\subsection{Character tables of $\SL_2(q)$ and $\PSL_2(q)$}\label{sec:conj-char}
The results in this subsection are well known; we include them for the reader's convenience, as the use of information from the character tables of these groups is ubiquitous throughout this article. 

Let $\tau$ be a generator of $\fq_{q^2}^\times$ and fix $\sigma=\tau^{q+1}$, $\tau_0=\tau^{q-1}$. Then define
\begin{align*}
    S(a)  = \begin{pmatrix}
        \sigma^a & \\ & \sigma^{-a}
    \end{pmatrix} \text{and, }
    T(b)  = \begin{pmatrix}
      ~ & -1\\ 1& \tau_0^{b}+\tau_0^{bq}
    \end{pmatrix}.
\end{align*}
Also if $\eta$ is a non-square in $\fq_q$, set $N'=\begin{pmatrix}
    1 & \eta \\ &1
\end{pmatrix}$ and set $N=\begin{pmatrix}
    1 & 1 \\ & 1
\end{pmatrix}.$ Fix $I=\begin{pmatrix}
        1 & \\ & 1
    \end{pmatrix}$. 
We consider the case when the field is of odd size. 
%The even case is discussed in the next section as $\SL_2(q)\cong\PSL_2(q)$ therein. 
Fix $\epsilon \in \mathbb{C}$ to be a primitive $q-1$ root of unity and $\eta_0 \in \mathbb{C}$ to be a primitive $q+1$ root of unity.
\subsubsection{The case of $\SL_2(q)$} Let $\delta=(-1)^{(q-1)/2}$, $\omega=\dfrac{1+\sqrt{\delta q}}{2}$, and $\omega^*=\dfrac{1-\sqrt{\delta q}}{2}$. We have the following table for $q\geq 5$; see \cite[p. 106]{GeMa20}.
\begin{table}[H]
  \setcellgapes{5pt}\makegapedcells
  \centering%
  \begin{tabular}{*{9}{c}}
    \hline
    $x^G$& $I$ & $-I$ & $N$ & $N'$ & $-N$ &$-N'$&\makecell{$S(a)$} & \makecell{$T(b)$}\\
\hline
    $\psi_1$& $1$ & $1$ & $1$ & $1$ & $1$ & $1$ & $1$ & $1$\\
    $\psi_q$& $q$ & $q$ & $\cdot$ & $\cdot$ & $\cdot$ & $\cdot$ & $1$ & $-1$\\
    $\psi_{q+1}^{(i)}$& $q+1$ & $(-1)^i(q+1)$ & $1$ & $1$ & $(-1)^i$ & $(-1)^i$ & $\e^{ai}+\e^{-ai}$ & $\cdot$\\
    $\psi_{q-1}^{(j)}$& $q-1$ & $(-1)^j(q-1)$ & $-1$ & $-1$ & $(-1)^{j+1}$ & $(-1)^{j+1}$ & $\cdot$ & $\eta_0^{bj}+\eta_0^{-bj}$\\
    $\psi_+'$& $\dfrac{q+1}{2}$ & $\dfrac{\delta(q+1)}{2}$ & $\omega$ & $\omega^*$ & $\delta\omega$ & $\delta\omega^*$ & $(-1)^a$ & $\cdot$\\
    $\psi_+''$& $\dfrac{q+1}{2}$ & $\dfrac{\delta(q+1)}{2}$ & $\omega^*$ & $\omega$ & $\delta\omega^*$ & $\delta\omega$ & $(-1)^a$ & $\cdot$\\
    $\psi_-'$& $\dfrac{q-1}{2}$ & $-\dfrac{\delta(q-1)}{2}$ & $-\omega^*$ & $-\omega$ & $\delta\omega^*$ & $\delta\omega$ & $\cdot$ & $(-1)^{b+1}$\\
    $\psi_-''$& $\dfrac{q-1}{2}$ & $-\dfrac{\delta(q-1)}{2}$ & $-\omega$ & $-\omega^*$ & $\delta\omega$ & $\delta\omega^*$ & $\cdot$ & $(-1)^{b+1}$\\
     \hline
  \end{tabular}
      \caption{Character table of $\SL_2(q)$, }\label{table-sl(q)}
\end{table}In \Cref{table-sl(q)}, we have the parameters $1\leq i\leq\dfrac{q-3}{2}$ and $1\leq j\leq\dfrac{q-1}{2}$.
\subsubsection{The case of $\PSL_2(q)$} First, we look at the case when $q$ is a power of $2$.
Then the conjugacy classes of $\PSL_2(q)=\SL_2(q)$ have representatives as
\begin{align*}
    I=\begin{pmatrix}
        1 & \\ & 1
    \end{pmatrix},
      N=  \begin{pmatrix}
        1 & 1\\ & 1
    \end{pmatrix},
         S(a), T(b),
\end{align*}
where $1\leq a\leq \frac{q}{2}-1$, $1\leq b\leq \frac{q}{2}$. Since there are $q+1$ conjugacy classes, $\PSL_2(q)$ has $q+1$ irreducible representations up to equivalence. We have the representations 
$\psi_1$ and $\psi_q$ of dimensions $1$ and $q$ respectively. For each 
$1\leq k\leq \frac{q}{2}-1$ we have representations $\psi^{(k)}_{q+1}$ of dimension
$q+1$ and for each $1\leq j\leq \frac{q}{2}$, we have representations
$\psi^{(j)}_{q-1}$ of dimension $q-1$. The description of these representations 
can be found in \cite[pp. 104-105]{GeMa20}. We present the character table in \cref{table-psl(2)}. 
\begin{table}[H]
\setcellgapes{5pt}\makegapedcells
  \centering%
  \begin{tabular}{*{5}{c}}
    \hline
     &    &     & $1\leq a\leq \frac{q}{2}-1$ & $1\leq b\leq \frac{q}{2}$\\
    $x$& $I$ & $N$ & $S(a)$ & $T(b)$ \\
     $|x^G|$ & $1$ & $q^2-1$ & $q(q+1)$ & $q(q-1)$\\
     \hline
    $\psi_1$ & $1$ & $1$ & $1$ & $1$ \\
    % \hline
    $\psi_q$ & $q$ & $\cdot$ & $1$ & $-1$ \\
    % \hline
    $\psi_{q+1}^{(k)}$ & $q+1$ & $1$ & $\epsilon^{ak}+\epsilon^{-ak}$ &  $\cdot$\\
    % \hline
    $\psi_{q-1}^{(j)}$ & $q-1$ & $-1$ &  $\cdot$ & $-\eta_0^{bj}-\eta_0^{-bj}$ \\
    \hline
  \end{tabular}
    \caption{Character table of $\PSL_2(2^m)$}\label{table-psl(2)}
\end{table}
Similarly when $q$ is odd, the group $\PSL_2(q)$ has $\frac{q+5}{2}$ conjugacy classes. Hence $\PSL_2(q)$ has $\frac{q+5}{2}$ irreducible representations upto equivalence. These representations are obtained by looking at non-faithful irreducible representations of $\SL_2(q)$.

In \cref{table-psl(q)34} we present the character table of $\PSL_2(q)$ in case $q\equiv 3\pmod{4}$, where we have $k=2,4,\cdots, \frac{q-3}{2}$, $j=2,4,\cdots, \frac{q-3}{2}$, $\omega=\frac{1+\sqrt{-q}}{2}$ and $\omega^*=\frac{1-\sqrt{-q}}{2}$.
\begin{table}[H]
  \setcellgapes{5pt}\makegapedcells
  \centering%
  \begin{tabular}{*{7}{c}}
    \hline
    & & & & $1\leq a\leq \dfrac{q-3}{4}$ & $1\leq b\leq\dfrac{q-3}{4}$ &\\
    $x$& $I$ & $N$ & $N'$ & \makecell{$S(a)$} & \makecell{$T(b)$} & $T\left(\frac{q+1}{4}\right)$ \\
    $|x^G|$& $1$ & $\dfrac{q^2-1}{2}$ & $\dfrac{q^2-1}{2}$ & $q(q+1)$ & ${q(q-1)}$ & $\dfrac{q(q-1)}{2}$ \\
     \hline
    $\psi_1$ & $1$ & $1$ & $1$ & $1$ & $1$ & $1$ \\
    $\psi_q$ & $q$ & $\cdot$ & $\cdot$ & $1$ & $-1$ & $-1$ \\
    $\psi_{q+1}^{(k)}$ & $q+1$ & $1$ & $1$ & $\epsilon^{ak}+\epsilon^{-ak}$ & $\cdot$ & $\cdot$\\
    $\psi_{q-1}^{(j)}$ & $q-1$ & $-1$ & $-1$ & $\cdot$ & $-\eta_0^{bj}-\eta_0^{-bj}$ & $-2\eta_0^{\frac{(q+1)j}{4}}$ \\
    $\psi_-'$ & $\frac{q-1}{2}$ & $-\omega^*$ & $-\omega$ & $\cdot$ & $(-1)^{b+1}$ & $(-1)^{\frac{q+5}{4}}$ \\
    $\psi_-''$ & $\frac{q-1}{2}$ & $-\omega$ & $-\omega^*$ & $\cdot$ & $(-1)^{b+1}$ & $(-1)^{\frac{q+5}{4}}$ \\
    \hline
  \end{tabular}
      \caption{Character table of $\PSL_2(q)$, $q\equiv3\pmod{4}$}\label{table-psl(q)34}
\end{table}
\begin{table}[H]
  \setcellgapes{5pt}\makegapedcells
  \centering%
  \begin{tabular}{*{7}{c}}
    \hline
    &&&& $1\leq a\leq \dfrac{q-5}{4}$ && $1\leq b\leq \dfrac{q-1}{4}$\\
    $x$& $I$ & $N$ & $N'$ & \makecell{$S(a)$} & $S\left(\frac{q-1}{4}\right)$ & \makecell{$T(b)$} \\
    $|x^G|$& $1$ & $\dfrac{q^2-1}{2}$ & $\dfrac{q^2-1}{2}$ & $q(q+1)$ & $\dfrac{q(q+1)}{2}$ & $q(q-1)$ \\
     \hline
    $\psi_1$ & $1$ & $1$ & $1$ & $1$ & $1$ & $1$ \\
    $\psi_q$ & $q$ & $\cdot$ & $\cdot$ & $1$ & $1$ & $-1$ \\
    $\psi_{q+1}^{(k)}$ & $q+1$ & $1$ & $1$ & $\epsilon^{ak}+\epsilon^{-ak}$ & $2\epsilon^{\frac{(q-1)k}{4}}$ & $\cdot$\\
        $\psi_{q-1}^{(j)}$ & $q-1$ & $-1$ & $-1$ & $\cdot$ & $\cdot$ & $-\eta_0^{bj}-\eta_0^{-bj}$ \\
    $\psi_+'$ & $\frac{q+1}{2}$ & $\omega$ & $\omega^*$ & $(-1)^a$ & $(-1)^{\frac{q-1}{4}}$ & $\cdot$ \\
    $\psi_+''$ & $\frac{q+1}{2}$ & $\omega^*$ & $\omega$ & $(-1)^a$ & $(-1)^{\frac{q-1}{4}}$ & $\cdot$ \\
    \hline
  \end{tabular}

  \caption{Character table of $\PSL_2(q)$, $q\equiv1\pmod{4}$}\label{table-psl(q)14}
\end{table}
In \cref{table-psl(q)14} we present the character table of $\PSL_2(q)$ in case $q\equiv 1\pmod{4}$, where we have $k=2,4,\cdots, \frac{q-5}{2}$, $j=2,4,\cdots, \frac{q-1}{2}$, $\omega=\frac{1+\sqrt{q}}{2}$ and $\omega^*=\frac{1-\sqrt{q}}{2}$.

\section{The results for the group $\PSL_2(q)$}\label{sec:prelim}

In this section, we present proofs of Theorems~\ref{thm:wedderburn PSL(even)} and~\ref{thm:wedderburn PSL(odd)}. We first determine the number of pairwise non-isomorphic simple $\Q G$-modules of each possible dimension of $\PSL_2(q)$, as these are central to the proofs of the theorems. We begin with the following results.
\begin{lemma}\cite{Janusz} \label{lemma:schurindexPSL}
     Let $G=\PSL_2(q)$, where $q=p^m$ for some prime $p$. Then $m_{\Q}(\chi)=1$ for each $\chi \in \irr(G)$.
\end{lemma}

\begin{lemma} \label{lemma:uniqueness}
Let $q=p^{n}$ be a prime power with $n\ge1$. Suppose $A=\{\, 1<d\le q+1 \, : \, d\mid q+1 \, \}$ and $B=\{\, 1<e\le q-1 \, : \, e\mid q-1 \,\}$. Define 
\[
S_A=\left\{\frac{(q-1)}{2}\varphi(d): d\in A\right\}\quad \text{and} \quad
S_B=\left\{\frac{(q+1)}{2}\varphi(e): e\in B\right\},
\]
where $\varphi(m)$ denotes the number of integers in $[1, m]$ coprime to $m$. Then $S_A\cap S_B=\varnothing$.
\end{lemma}

\begin{proof}
Suppose, for contradiction, that there exist $d\in A$ and $e\in B$ such that
\begin{equation} \label{eq:1}
\frac{(q-1)\varphi(d)}{2}=\frac{(q+1)\varphi(e)}{2}. 
\end{equation}

Let $g=\gcd(q-1,q+1)$. Then for a prime power $q=p^{m}$, we have
\[
g=\begin{cases}
1& q \text{ even},\\
2& q \text{ odd}.
\end{cases}
\]
Write $q-1=ga$ and $q+1=gb$ with $\gcd(a,b)=1$. Equation~\eqref{eq:1} becomes
\[
a\varphi(d)=b\varphi(e). 
\]
Since $\gcd(a,b)=1$, there exists an integer $k\ge1$ such that $\varphi(d)=bk$ and $\varphi(e)=ak$. Because $\varphi(d)<d\le q+1=gb$ and $\varphi(e)<e\le q-1=ga$, we obtain $bk<gb$ and $ak<ga$, hence $k<g$. If $g=1$, this is impossible since $k\ge1$. Thus no solutions exist in this case. Therefore, $g=2$, so $q$ is odd and $k=1$. Hence, $\varphi(d)=\frac{q+1}{2}$ and $\varphi(e)=\frac{q-1}{2}$. Note that $\varphi(n)=n/2$ if and only if $n=2^r$ for some positive integer $r$. In this way, we obtain $d=q+1=2^k$ for some positive integer $k$ and $e=q-1=2^{k'}$ for some positive integer $k'$ which is impossible since they differ by $2$.
% \[
% d=q+1=2^k \,\,\,\, \text{for some positive integer}\,\,\,\, \frac{q+1}{2} \,\,\,\, \text{an odd prime}\quad \text{and} \quad
% e=q-1 \,\,\,\, \text{with} \,\,\,\, \frac{q-1}{2} \,\,\,\, \text{an odd prime}.
% \]
% Thus, both $q-1$ and $q+1$ are twice an odd prime, which is impossible since they differ by $2$.
This contradiction shows that no such $d$ and $e$ exist. Therefore, $S_A\cap S_B=\varnothing$.
\end{proof}

Lemma~\ref{lemma:RationRepsPSL2} determines the number of pairwise non-isomorphic simple $\Q G$-modules of each possible dimension of $\PSL_2(q)$, where $q$ is a power of $2$.
\begin{lemma} \label{lemma:RationRepsPSL2}
   Let $G=\PSL_2(q)$, where $q$ is a power of $2$. For a positive integer $m$, let $\varphi(m)$ denote the number of integers in $[1, m]$ that are coprime to $m$, and let $\tau(m)$ denote the number of positive divisors of $m$. Define
\[
A = \{ \, 2 < d \leq q+1 \, : \, d \mid q+1 \, \} \,\,\,\, \text{and} \,\,\,\, 
B = \{ \, 2 < d \leq q-1 \, : \, d \mid q-1 \, \}.
\]
Then the following statements hold.
\begin{enumerate}
    \item There is a unique trivial $\Q G$-module of dimension $1$.

    \item There is a unique (up to isomorphism) Steinberg $\Q G$-module of dimension $q$.

    \item For each $d \in A$, there exists a unique (up to isomorphism) simple $\Q G$-module of dimension $\frac{q-1}{2}\varphi(d)$, and these modules are pairwise non-isomorphic for distinct $d$.

    \item For each $d \in B$, there exists a unique (up to isomorphism) simple $\Q G$-module of dimension $\frac{q+1}{2}\varphi(d)$, and these modules are pairwise non-isomorphic for distinct $d$.

    \item $|\irr_{\Q}(G)|=\tau(q-1) + \tau(q+1)$.
\end{enumerate}
\end{lemma}

\begin{proof}
Note that there exists a unique trivial $\Q G$-module of dimension $1$, and there exists a unique (up to isomorphism) Steinberg $\Q G$-module of dimension $q$. Next, from Table~\ref{table-psl(2)}, we observe that
\begin{align*}
|\Gal(\Q(\psi^{(j)}_{q-1})/\Q)|
&= [\Q(\psi^{(j)}_{q-1}) : \Q] \\
&= [\Q(\eta_0^{j}+\eta_0^{-j}) : \Q] \\
&= \frac{1}{2}\varphi\!\left(\frac{q+1}{\gcd(q+1,j)}\right)
\end{align*}
as $\eta_0$ is a primitive $(q+1)$-th root of unity in $\mathbb{C}$. Furthermore, $j \in \left\{1,2,\ldots,\frac{q}{2}\right\}$. Therefore, all irreducible complex characters of $\PSL_2(q)$ of type $\psi^{(j)}_{q-1}$, i.e.,
\[
\{\chi \in \irr(\PSL_2(q)) : \chi(1)=q-1\},
\]
form exactly one distinct Galois conjugacy class of size $\frac{1}{2}\varphi(d)$ for each
\[
d \in \{\, 2<d\le q+1 \, : \,  d \mid q+1 \,\}.
\]
Hence, the assertion that there exists a unique (up to isomorphism) simple $\Q G$-module of dimension $\frac{q-1}{2}\varphi(d)$ for each $d \in A$ follows directly from Lemmas~\ref{lemma:schurindexPSL} and~\ref{lemma:uniqueness} together with Schur theory.

Similarly, from Table~\ref{table-psl(2)}, we have
\begin{align*}
|\Gal(\Q(\psi^{(k)}_{q+1})/\Q)|
&= [\Q(\psi^{(k)}_{q+1}) : \Q] \\
&= [\Q(\epsilon^{k}+\epsilon^{-k}) : \Q] \\
&= \frac{1}{2}\varphi\!\left(\frac{q-1}{\gcd(q-1,k)}\right)
\end{align*}
as $\epsilon$ is a primitive $(q-1)$-th root of unity in $\mathbb{C}$. Furthermore, $k \in \left\{1,2,\ldots,\frac{q}{2}-1\right\}$. Therefore, all irreducible complex characters of $\PSL_2(q)$ of type $\psi^{(k)}_{q+1}$, i.e.,
\[
\{\chi \in \irr(\PSL_2(q)) : \chi(1)=q+1\},
\]
form exactly one distinct Galois conjugacy class of size $\frac{1}{2}\varphi(d)$ for each
\[
d \in \{\, 2<d\le q-1 \, : \,  d \mid q-1 \, \}.
\]
Hence, the assertion that there exists a unique (up to isomorphism) simple $\Q G$-module of dimension $\frac{q+1}{2}\varphi(d)$ for each $d \in B$ follows directly from Lemmas~\ref{lemma:schurindexPSL} and~\ref{lemma:uniqueness} together with Schur theory.

Finally, observe that there exists a unique (up to isomorphism) simple $\Q G$-module corresponding to each divisor greater than $1$ of $q-1$ and $q+1$, and there exists a unique (up to isomorphism) simple $\Q G$-module of dimensions $1$ and $q$, respectively, as discussed above. Moreover, these rational representations exhaust all $\chi \in \irr(G)$. Hence, part~(5) follows. This completes the proof of Lemma~\ref{lemma:RationRepsPSL2}.
\end{proof}

We now proceed to the proof of Theorem~\ref{thm:wedderburn PSL(even)}.
\begin{proof}[Proof of Theorem~\ref{thm:wedderburn PSL(even)}]
    Let $G=\PSL_2(q)$, where $q$ is power of $2$, and let $\chi \in \irr(G)$. Suppose $\rho$ is an irreducible $\mathbb{Q}$-representation of $G$ affording the character $\Omega(\chi)$. Denote by $A_\mathbb{Q}(\chi)$ the simple component in the Wedderburn decomposition of $\mathbb{Q}G$ corresponding to $\rho$, so that $A_\mathbb{Q}(\chi) \cong M_n(D)$ for some $n \in \mathbb{N}$ and a division algebra $D$. By Lemma~\ref{lemma:schurindexPSL}, we have $m_\mathbb{Q}(\chi) = 1$. Moreover, Lemma~\ref{Reiner} shows that $[D:Z(D)] = m_\mathbb{Q}(\chi)^2$ and $Z(D) = \mathbb{Q}(\chi)$. Hence, $D = Z(D) = \mathbb{Q}(\chi)$. Next, consider
			\[
			\rho \cong \bigoplus_{i=1}^l \rho_i,
			\]
			where $l = [\mathbb{Q}(\chi): \mathbb{Q}]$ and each $\rho_i$ is an irreducible complex representation of $G$ affording the character $\chi^{\sigma_i}$ for some $\sigma_i \in \mathrm{Gal}(\mathbb{Q}(\chi)/\mathbb{Q})$. Since $m_\mathbb{Q}(\chi) = 1$, it follows from \cite[Theorem 3.3.1]{JR} that $n = \chi(1)$.

			Furthermore, we have
			\[
			\cd(G) = \{1, q-1, q, q+1\}.
			\]
    Note that there is only one trivial representation, which is of degree one complex representation. Therefore, the simple component in the Wedderburn decomposition of $\mathbb{Q}G$ corresponding to the irreducible $\mathbb{Q}$-representations of $G$ affording the character $\Omega(\chi)$ with $\chi(1)=1$ is isomorphic to $\Q$.
    
    Next, let $A=\{\, 2 < d \leq q+1 \, : \, d \mid q+1 \, \}$. By Lemma~\ref{lemma:RationRepsPSL2}, for each $d \in A$, there is exactly one irreducible $\Q$-representation of $G$ that affords the character $\Omega(\chi)$ with $\chi(1)=q-1$ and $\Q(\chi)=\Q(\zeta_{d}+\zeta_d^{-1})$. Moreover, these are all irreducible $\Q$-representation of $G$ that affords the character $\Omega(\chi)$ with $\chi(1)=q-1$. Therefore, the direct sum of the simple components in the Wedderburn decomposition of $\mathbb{Q}G$ corresponding to all inequivalent irreducible $\mathbb{Q}$-representations of $G$ affording the character $\Omega(\chi)$ with $\chi(1)=q-1$ is isomorphic to 
    $$\bigoplus_{d \in A} \M_{q-1}(\Q(\zeta_{d}+\zeta_d^{-1})).$$
    Analogous to the similar argument as above, by Lemma~\ref{lemma:RationRepsPSL2}, the direct sum of the simple components in the Wedderburn decomposition of $\mathbb{Q}G$ corresponding to all inequivalent irreducible $\mathbb{Q}$-representations of $G$ affording the character $\Omega(\chi)$ with $\chi(1)=q+1$ is isomorphic to 
    $$\bigoplus_{d \in B} \M_{q+1}(\Q(\zeta_{d}+\zeta_d^{-1})),$$
    where $B=\{ \, 2 < d \leq q-1 \, : \, d \mid q-1 \, \}$.

    Finally, note that there is only one irreducible complex representation of $G$ of degree $q$ that affords the character $\psi_q$. Moreover $\Q(\psi_q)=\Q$. Therefore, the simple component in the Wedderburn decomposition of $\mathbb{Q}G$ corresponding to the irreducible $\mathbb{Q}$-representations of $G$ affording the character $\Omega(\psi_q)$ is isomorphic to $\M_q(\Q)$. 

    Hence, by collecting the simple components corresponding to all inequivalent irreducible $\mathbb{Q}$-representations of $G$ affording the character $\Omega(\chi)$ with $\chi \in \irr(G)$, the result follows. This completes the proof of Theorem~\ref{thm:wedderburn PSL(even)}.  
\end{proof}

Next, we consider the case where $q$ is a power of an odd prime $p$. Lemma~\ref{lemma:RationRepsPSL2 odd} determines the number of pairwise non-isomorphic simple $\mathbb{Q}G$-modules of each possible dimension for $G = \mathrm{PSL}_2(q)$, where $q$ is a power of an odd prime $p$.

\begin{lemma} \label{lemma:RationRepsPSL2 odd}
    Let $G=\PSL_2(q)$, where $q$ is a power of an odd prime $p$. For a positive integer $m$, let $\varphi(m)$ denote the number of integers in $[1,m]$ that are coprime to $m$, and let $\tau(m)$ denote the number of positive divisors of $m$. Define
\[
A=\left\{\, 2 < d \leq \tfrac{q+1}{2} \, : \, d \mid \tfrac{q+1}{2} \,\right\}
\,\,\,\, \text{and} \,\,\,\,
B=\left\{\, 2 < d \leq \tfrac{q-1}{2} \, : \, d \mid \tfrac{q-1}{2} \,\right\}.
\]
Then the following statements hold.
\begin{enumerate}
    \item There is a unique trivial $\Q G$-module of dimension $1$.

    \item There is a unique (up to isomorphism) Steinberg $\Q G$-module of dimension $q$.

    \item For each $d \in A$, there exists a unique (up to isomorphism) simple $\Q G$-module of dimension $\frac{q-1}{2}\varphi(d)$, and these modules are pairwise non-isomorphic for distinct $d$.

    \item For each $d \in B$, there exists a unique (up to isomorphism) simple $\Q G$-module of dimension $\frac{q+1}{2}\varphi(d)$, and these modules are pairwise non-isomorphic for distinct $d$.

    \item The remaining simple $\Q G$-modules are as follows, which are non-isomorphic to any modules arising from parts (3) and (4):
    \begin{itemize}
        \item If $q \equiv 3 \pmod{4}$, there exists a unique (up to isomorphism) simple $\Q G$-module of dimension $q-1$.
        \item If $q \equiv 1 \pmod{4}$ and $q$ is not a square, there exists a unique (up to isomorphism) simple $\Q G$-module of dimension $q+1$.
        \item If $q \equiv 1 \pmod{4}$ and $q$ is a square, there exist two (up to isomorphism) simple $\Q G$-modules of dimension $\frac{q+1}{2}$.
    \end{itemize}

    \item \[
|\irr_{\Q}(G)| =
\begin{cases}
\tau\!\left(\dfrac{q-1}{2}\right)
+\tau\!\left(\dfrac{q+1}{2}\right)
+1
& \text{if } q \equiv 1 \pmod{4} \text{ and } q \text{ is a square}, \\[1ex]
\tau\!\left(\dfrac{q-1}{2}\right)
+\tau\!\left(\dfrac{q+1}{2}\right)
& \text{otherwise}.
\end{cases}
\]
\end{enumerate}
\end{lemma}

\begin{proof}
From Tables~\ref{table-psl(q)34} and~\ref{table-psl(q)14}, we obtain
\begin{align*}
|\Gal(\Q(\psi^{(j)}_{q-1})/\Q)|
&= [\Q(\psi^{(j)}_{q-1}) : \Q] \\
&= [\Q(\eta_0^{j}+\eta_0^{-j}) : \Q] \\
&= \frac{1}{2}\varphi\!\left(\frac{q+1}{\gcd(q+1,j)}\right)
\end{align*}
since $\eta_0$ is a primitive $(q+1)$-th root of unity in $\mathbb{C}$. Moreover,
$j \in \{2,4,\dots,\tfrac{q-3}{2}\}$ when $q \equiv 3 \pmod{4}$, and
$j \in \{2,4,\dots,\tfrac{q-1}{2}\}$ when $q \equiv 1 \pmod{4}$.

For each such $j$, we have $\Q(\eta_0^{j}) = \Q(\zeta_d)$ for some
$d \in A=\left\{\, 2 < d \leq \tfrac{q+1}{2} \, : \, d \mid \tfrac{q+1}{2} \,\right\}$.
Note that $\tfrac{q+1}{2}$ is even when $q \equiv 3 \pmod{4}$ and odd when
$q \equiv 1 \pmod{4}$. It follows that all irreducible complex characters
of $\PSL_2(q)$ of type $\psi^{(j)}_{q-1}$, i.e.,
\[
\{\chi \in \irr(\PSL_2(q)) : \chi(1) = q-1\},
\]
split into Galois conjugacy classes indexed by $d \in A$, each of size
$\tfrac{1}{2}\varphi(d)$. Hence, by Lemmas~\ref{lemma:schurindexPSL} and~\ref{lemma:uniqueness},
together with Schur theory, there exists a unique (up to isomorphism)
simple $\Q G$-module of dimension $\tfrac{q-1}{2}\varphi(d)$ for each
$d \in A$. This proves part~(3).

The proofs of the remaining parts of the lemma follow by arguments
analogous to those used above and in the proof of
Lemma~\ref{lemma:RationRepsPSL2}. This completes the proof of
Lemma~\ref{lemma:RationRepsPSL2 odd}.
\end{proof}

Finally, we conclude this section by presenting the proof of Theorem~\ref{thm:wedderburn PSL(odd)}.
\begin{proof}[Proof of Theorem~\ref{thm:wedderburn PSL(odd)}]
    Let $G=\PSL_2(q)$, where $q=p^m$ for some odd prime $p$, and let $\chi \in \irr(G)$. Suppose $\rho$ is an irreducible $\mathbb{Q}$-representation of $G$ affording the character $\Omega(\chi)$. Denote by $A_\mathbb{Q}(\chi)$ the simple component in the Wedderburn decomposition of $\mathbb{Q}G$ corresponding to $\rho$, so that $A_\mathbb{Q}(\chi) \cong \M_n(D)$ for some $n \in \mathbb{N}$ and a division algebra $D$. By Lemma~\ref{lemma:schurindexPSL}, we have $m_\mathbb{Q}(\chi) = 1$. Moreover, Lemma~\ref{Reiner} shows that $[D:Z(D)] = m_\mathbb{Q}(\chi)^2$ and $Z(D) = \mathbb{Q}(\chi)$. Hence, $D = Z(D) = \mathbb{Q}(\chi)$. Next, consider
			\[
			\rho \cong \bigoplus_{i=1}^l \rho_i,
			\]
			where $l = [\mathbb{Q}(\chi): \mathbb{Q}]$ and each $\rho_i$ is an irreducible complex representation of $G$ affording the character $\chi^{\sigma_i}$ for some $\sigma_i \in \mathrm{Gal}(\mathbb{Q}(\chi)/\mathbb{Q})$. Since $m_\mathbb{Q}(\chi) = 1$, it follows from \cite[Theorem 3.3.1]{JR} that $n = \chi(1)$. The rest of the proof proceeds in the following two cases.

            \begin{enumerate}
                \item {\bf Case ($q\equiv3\pmod{4}$).} In this case, we have
\[
			\cd(G) = \left \{1, \frac{q-1}{2}, q-1, q, q+1 \right \}.
			\]
    Note that there is only trivial representation, which is of degree one complex representation. Therefore, the simple component in the Wedderburn decomposition of $\mathbb{Q}G$ corresponding to the irreducible $\mathbb{Q}$-representations of $G$ affording the character $\Omega(\chi)$ with $\chi(1)=1$ is isomorphic to $\Q$. Furthermore, we have
     $$\Q(\psi_-')=\Q(\psi_-'')=\Q(\omega)=\Q(\sqrt{-q})$$
    (see Table~\ref{table-psl(q)34}). Hence, the simple component in the Wedderburn decomposition of $\mathbb{Q}G$ corresponding to the irreducible $\mathbb{Q}$-representations of $G$ affording the character $\Omega(\chi)$ with $\chi(1)=\frac{q-1}{2}$ is isomorphic to $\M_{\frac{q-1}{2}}(\Q(\sqrt{-q}))$.
    
    Next, let $A=\left\{\, 2 < d \leq \tfrac{q+1}{2} \, : \, d \mid \tfrac{q+1}{2} \,\right\}$. By Lemma~\ref{lemma:RationRepsPSL2 odd}, for each $d \in A$, there is exactly one irreducible $\Q$-representation of $G$ that affords the character $\Omega(\chi)$ with $\chi(1)=q-1$ and $\Q(\chi)=\Q(\zeta_{d}+\zeta_d^{-1})$. Moreover, these are all irreducible $\Q$-representation of $G$ that affords the character $\Omega(\chi)$ with $\chi(1)=q-1$. Therefore, the direct sum of the simple components in the Wedderburn decomposition of $\mathbb{Q}G$ corresponding to all inequivalent irreducible $\mathbb{Q}$-representations of $G$ affording the character $\Omega(\chi)$ with $\chi(1)=q-1$ is isomorphic to 
    $$\bigoplus_{d \in A} \M_{q-1}(\Q(\zeta_{d}+\zeta_d^{-1})).$$
    Analogous to the similar argument as above, by Lemma~\ref{lemma:RationRepsPSL2 odd}, the direct sum of the simple components in the Wedderburn decomposition of $\mathbb{Q}G$ corresponding to all inequivalent irreducible $\mathbb{Q}$-representations of $G$ affording the character $\Omega(\chi)$ with $\chi(1)=q+1$ is isomorphic to 
    $$\bigoplus_{d \in B} \M_{q+1}(\Q(\zeta_{d}+\zeta_d^{-1})),$$
    where $B=\left\{\, 2 < d \leq \tfrac{q-1}{2} \, : \, d \mid \tfrac{q-1}{2} \,\right\}$.

    Finally, note that that there is only one irreducible complex representation of $G$ of degree $q$ that affords the character $\psi_q$. Moreover $\Q(\psi_q)=\Q$. Therefore, the simple component in the Wedderburn decomposition of $\mathbb{Q}G$ corresponding to the irreducible $\mathbb{Q}$-representations of $G$ affording the character $\Omega(\psi_q)$ is isomorphic to $\M_q(\Q)$. 

    Hence, by collecting the simple components corresponding to all inequivalent irreducible $\mathbb{Q}$-representations of $G$ affording the character $\Omega(\chi)$ with $\chi \in \irr(G)$, the result follows.

    \item {\bf Case ($q\equiv1\pmod{4}$).} In this case, we have
\[
\cd(G)=\left\{1,\frac{q+1}{2},q-1,q,q+1\right\}.
\] 

For the subcase ($q$ is not a square), we have
$$
\Q(\psi_+')=\Q(\psi_+'')=\Q(\omega)=\Q(\sqrt{q})
$$
(see Table~\ref{table-psl(q)14}). It follows by Lemma~\ref{lemma:RationRepsPSL2 odd} that the simple component in the Wedderburn decomposition of $\Q G$ corresponding to the irreducible $\Q$-representations of $G$ affording the character $\Omega(\chi)$ with $\chi(1)=\frac{q+1}{2}$ is isomorphic to $\M_{\frac{q+1}{2}}(\Q(\sqrt{q}))$.

For the subcase ($q$ is a square), we have
$$
\Q(\psi_+')=\Q(\psi_+'')=\Q(\omega)=\Q
$$
(see Table~\ref{table-psl(q)14}). Consequently, by Lemma~\ref{lemma:RationRepsPSL2 odd}, the simple components in the Wedderburn decomposition of $\Q G$ corresponding to the irreducible $\Q$-representations of $G$ affording the character $\Omega(\chi)$ with $\chi(1)=\frac{q+1}{2}$ are isomorphic to $2\M_{\frac{q+1}{2}}(\Q)$.

The remaining part of the proof proceeds by arguments similar to those in the previous case, and hence the results follow.
            \end{enumerate}
	This completes the proof of Theorem~\ref{thm:wedderburn PSL(odd)}.
\end{proof}

\section{The results for the group $\SL_2(q)$} \label{sec:SLresult}
In this section, we turn to the special linear group $G = \SL_2(q)$. Note that if $q$ is a power of $2$, then $\SL_2(q) = \PSL_2(q)$. Thus, we focus on the group $\SL_2(q)$, where $q$ is a power of an odd prime $p$. Here, we present the proof of Theorem~\ref{thm:wedderburn SL(odd)}. We again first determine the number of pairwise non-isomorphic simple $\Q G$-modules of each possible dimension of $\SL_2(q)$, as these are central to the proof of the theorem. Before proving Theorem~\ref{thm:wedderburn SL(odd)}, we state the following lemma, which determines the Schur index $m_{\Q}(\chi)$ for each $\chi \in \irr(G)$ over $\Q$.
 
\begin{lemma}\cite[Theorem~3.1]{Shahabi} \label{lemma:schurindexSL}
Let $G=\SL_2(q)$. If $q$ is a power of 2, then the Schur index of any irreducible complex character of $G$ over $\Q$ is 1. If $q$ is a power of an odd prime $p$, then the Schur indices of the irreducible complex characters of $G$ over $\Q$ are given in \Cref{table:SchurIndexSL2}.
\begin{table}[h]
\centering
\renewcommand{\arraystretch}{1.3}
\begin{tabular}{c|cc}
 & $q \equiv 1 \pmod{4}$ & $q \equiv 3 \pmod{4}$ \\
\hline
$\psi_1$ & $1$ & $1$ \\
% \hline
$\psi_q$ & $1$ & $1$ \\
% \hline
$\psi_{q+1}^{(i)}$, $i\equiv 1\pmod2$ & $2$ & $2$ \\
$\psi_{q+1}^{(i)}$, $i\equiv 0\pmod 2$ & $1$ & $1$ \\
% \hline
$\psi_{q-1}^{(j)}$, $j\equiv 1\pmod2$ & $2$ & $2$ \\
$\psi_{q-1}^{(j)}$, $j\equiv 0\pmod 2$ & $1$ & $1$ \\
% \hline
$\psi_+'$ & $1$ & $1$ \\
$\psi_+''$ & $1$ & $1$ \\
% \hline
$\psi_-'$ & $2$ & $1$ \\
$\psi_-''$ & $2$ & $1$ \\
\hline
\end{tabular}
\caption{Table of Schur Indices} \label{table:SchurIndexSL2}
\end{table}
\end{lemma}

We now prove Lemma~\ref{lemma:RationRepsSL2}, which determines the number of pairwise non-isomorphic simple $\mathbb{Q}G$-modules of each possible dimension for $G = \mathrm{SL}_2(q)$, where $q$ is a power of an odd prime $p$.
\begin{lemma} \label{lemma:RationRepsSL2}
    Let $G=\SL_2(q)$, where $q$ is a power of an odd prime $p$. For a positive integer $d$, let $\varphi(d)$ denote the number of integers in $[1,d]$ that are coprime to $d$, and let $\tau(d)$ denote the number of positive divisors of $d$. Define
    \begin{align*}
        & A=\left\{\, 2 < m \leq \tfrac{q+1}{2} \, : \, m \mid \tfrac{q+1}{2} \, \right\}, \\
        & A'=\left\{\, 2 < m \leq q+1 \, : \, m \mid (q+1), \ m \nmid \tfrac{q+1}{2} \, \right\},\\
        & B=\left\{\, 2 < m \leq \tfrac{q-1}{2} \, : \, m \mid \tfrac{q-1}{2} \, \right\}, \, \, \, \,\text{and}\\
        & B'=\left\{\, 2 < m \leq q-1 \, : \, m \mid (q-1), \ m \nmid \tfrac{q-1}{2}  \, \right\}.
    \end{align*}
Then the following statements hold.
\begin{enumerate}
    \item There is a unique trivial $\Q G$-module of dimension $1$.

    \item There is a unique (up to isomorphism) Steinberg $\Q G$-module of dimension $q$.

    \item For each $d \in A$, there exists a unique (up to isomorphism) simple $\Q G$-module of dimension $\frac{q-1}{2}\varphi(d)$, and these modules are pairwise non-isomorphic for distinct $d$.

    \item For each $d \in A'$, there exists a unique (up to isomorphism) simple $\Q G$-module of dimension $(q-1)\varphi(d)$, and these modules are pairwise non-isomorphic for distinct $d$, and are non-isomorphic to any modules arising from part (3).

    \item For each $d \in B$, there exists a unique (up to isomorphism) simple $\Q G$-module of dimension $\frac{q+1}{2}\varphi(d)$, and these modules are pairwise non-isomorphic for distinct $d$.

    \item For each $d \in B'$, there exists a unique (up to isomorphism) simple $\Q G$-module of dimension $(q+1)\varphi(d)$, and these modules are pairwise non-isomorphic for distinct $d$, and are non-isomorphic to any modules arising from part (5).

    \item The remaining simple $\Q G$-modules arising from the principal series are as follows, which are non-isomorphic to any modules arising from parts (3)-(6):
    \begin{itemize}
        \item If $q \equiv 3 \pmod{4}$, there exists a unique (up to isomorphism) simple $\Q G$-module of dimension $2(q-1)$.
        \item If $q \equiv 1 \pmod{4}$ and $q$ is not a square, there exists a unique (up to isomorphism) simple $\Q G$-module of dimension $2(q-1)$.
        \item If $q \equiv 1 \pmod{4}$ and $q$ is a square, there exist two (up to isomorphism) simple $\Q G$-modules of dimension $q-1$.
    \end{itemize}

    \item The remaining simple $\Q G$-modules arising from the complementary series are as follows, which are non-isomorphic to any modules arising from parts (3)-(7):
    \begin{itemize}
        \item If $q \equiv 3 \pmod{4}$, there exists a unique (up to isomorphism) simple $\Q G$-module of dimension $q+1$.
        \item If $q \equiv 1 \pmod{4}$ and $q$ is not a square, there exists a unique (up to isomorphism) simple $\Q G$-module of dimension $q+1$.
        \item If $q \equiv 1 \pmod{4}$ and $q$ is a square, there exist two (up to isomorphism) simple $\Q G$-modules of dimension $\frac{q+1}{2}$.
    \end{itemize}

    \item \[
|\irr_{\Q}(G)| =
\begin{cases}
\tau(q-1) + \tau(q+1) + 2& 
\text{if } q \equiv 1 \pmod{4} \text{ and } q \text{ is a square}, \\[1ex]
\tau(q-1) + \tau(q+1) & 
\text{otherwise}.
\end{cases}
\]
\end{enumerate}
\end{lemma}
\begin{proof}
    From Table~\ref{table-sl(q)}, we get
\begin{align*}
|\Gal(\Q(\psi^{(j)}_{q-1})/\Q)|
&= [\Q(\psi^{(j)}_{q-1}) : \Q] \\
&= [\Q(\eta_0^{j}+\eta_0^{-j}) : \Q] \\
&= \frac{1}{2}\varphi\!\left(\frac{q+1}{\gcd(q+1,j)}\right)
\end{align*}
since $\eta_0$ is a primitive $(q+1)$-th root of unity in $\mathbb{C}$ and $1 \leq j \leq \frac{q-1}{2}$.

For each such $j$ with $j \equiv 0 \pmod{2}$, we have $\Q(\eta_0^{j}) = \Q(\zeta_d)$ for some
$d \in A=\left\{\, 2 < d \leq \tfrac{q+1}{2} \, : \, d \mid \tfrac{q+1}{2} \, \right\}$.
Note that $\tfrac{q+1}{2}$ is even when $q \equiv 3 \pmod{4}$ and odd when
$q \equiv 1 \pmod{4}$. It follows that all irreducible complex characters
of $\SL_2(q)$ of type $\psi^{(j)}_{q-1}$ with $j \equiv 0 \pmod{2}$ split into Galois conjugacy classes indexed by $d \in A$, each of size $\tfrac{1}{2}\varphi(d)$. Futrthermore, the Schur index of such character $\chi = \psi^{(j)}_{q-1}$ with $j \equiv 0 \pmod{2}$ over the field of rational numbers $m_{\Q}(\chi)=1$ (see Lemma~\ref{lemma:schurindexSL}).  Hence, by  Schur theory, there exists a unique (up to isomorphism)
simple $\Q G$-module of dimension $\tfrac{q-1}{2}\varphi(d)$ for each
$d \in A$. This proves part~(3). 

Similarly, for each such $j$ with $j \equiv 1 \pmod{2}$, we have $\Q(\eta_0^{j}) = \Q(\zeta_d)$ for some
$$d \in A'=\left\{\, 2 < d \leq q+1 \, : \, d \mid (q+1), \ d \nmid \tfrac{q+1}{2} \, \right\}.$$
Note that $\tfrac{q+1}{2}$ is even when $q \equiv 3 \pmod{4}$ and odd when
$q \equiv 1 \pmod{4}$. It follows that all irreducible complex characters
of $\SL_2(q)$ of type $\psi^{(j)}_{q-1}$ with $j \equiv 0 \pmod{2}$ split into Galois conjugacy classes indexed by $d \in A'$, each of size $\tfrac{1}{2}\varphi(d)$. Futrthermore, the Schur index of such character $\chi = \psi^{(j)}_{q-1}$ with $j \equiv 0 \pmod{2}$ over the field of rational numbers $m_{\Q}(\chi)=2$ (see Lemma~\ref{lemma:schurindexSL}).  Hence, by  Schur theory, there exists a unique (up to isomorphism)
simple $\Q G$-module of dimension $(q-1) \varphi(d)$ for each
$d \in A'$. This proves part~(4).

Observe that Schur theory ensures the existence of a unique irreducible representation $\rho$ of $\SL_2(q)$ over $\mathbb{Q}$ such that, for each representative $\chi \in \irr(\SL_2(q))$ of a distinct Galois conjugacy class, the character $\chi$ occurs as an irreducible constituent of $\rho \otimes_{\mathbb{Q}} \mathbb{C}$ with multiplicity $m_{\mathbb{Q}}(\chi)$. In particular, 
\[
\rho(1) = m_{\mathbb{Q}}(\chi)\,[\mathbb{Q}(\chi) : \mathbb{Q}]\,\chi(1).
\]

We omit the remaining details, as the proofs of the other parts of the lemma follow by arguments analogous to those above, together with Lemma~\ref{lemma:schurindexSL}. This completes the proof of Lemma~\ref{lemma:RationRepsSL2}.
\end{proof}

% Note that the procedure used to identify a division algebra is to compute its
% Hasse invariants at the primes of its center.
% We denote by $D_d$ a division ring such that $D_d$ is a quaternion algebra with
% \[
% Z(D_d) = \Q(\zeta_d + \zeta_d^{-1}),
% \]
% where $\zeta_d$ is a primitive $d$-th root of unity. 
% The algebra $D_d$ has nonzero Hasse invariant at each real prime of its center, and all other invariants are zero except in the following cases:
% \begin{itemize}
%     \item If $d = s^m$ or $d = 2 s^m$, where $s$ is a prime with $s \equiv -1 \pmod{4}$ and $m>1$, then the Hasse invariant is nonzero at the unique prime $\mathfrak{p}$ of $Z(D_d)$ lying above $s$.
%     \item If $d = 4$, then the Hasse invariant is nonzero at the prime ideal $(2)$ of $\Z$.
% \end{itemize}

% Next, we denote by $D'$ a division ring such that $D'$ is a quaternion algebra with
% \[
% Z(D') = \Q(\sqrt{q}),
% \]
% where $q > 0$ is squarefree. The algebra $D'$ has nonzero Hasse invariant at both real primes of $\Q(\sqrt{q})$, and zero Hasse invariant at all other primes.

% Finally, we denote by $D''$ a division ring such that $D''$ is a quaternion algebra with
% \[
% Z(D'') = \Q.
% \]
% The algebra $D''$ has nonzero Hasse invariant at the real prime of $\Q$ and at the finite prime $(p)$ dividing $q$, and zero Hasse invariant at all other primes.

Let $\mathbb{F}$ be a field with $\operatorname{char}(\mathbb{F}) \neq 2$, and let $a,b \in \mathbb{F}^\times$. The \emph{quaternion algebra} over $\mathbb{F}$ associated to $(a,b)$, denoted by $(a,b)_{\mathbb{F}}$, is the $\mathbb{F}$-algebra
\[
(a,b)_{\mathbb{F}} \;=\; \mathbb{F}\langle i,j \rangle \big/ \big(i^2 - a,\; j^2 - b,\; ij + ji\big).
\]
Equivalently, $(a,b)_{\mathbb{F}}$ is the $4$-dimensional $\mathbb{F}$-algebra with basis $\{1,i,j,ij\}$ satisfying
\[
i^2 = a,\quad j^2 = b,\quad ij = -ji.
\]
In particular, $(a,b)_{\mathbb{F}}$ is a central simple algebra over $\mathbb{F}$ of dimension $4$. Moreover, $(a,b)_{\mathbb{F}}$ is \emph{split}, i.e., $(a,b)_{\mathbb{F}} \cong \M_2(\mathbb{F})$, if and only if the quadratic form
\[
x_0^2 - a x_1^2 - b x_2^2 + ab x_3^2
\]
has a nontrivial zero over $\mathbb{F}$. Otherwise, $(a,b)_{\mathbb{F}}$ is a division algebra.
We now compute the quaternion algebras appearing in the decomposition of the group algebra $\Q[\SL_2(q)]$. 
From \cite{Janusz}, one knows the quaternion algebras in cyclic algebra notation. 
For example, according to \cite[Lemma 4]{Janusz}, one has that the $4$-dimensional cyclic algebra attached to the character $\psi_{q-1}^{(j)}$ is given by $(\Q(\eta_0^{j}),\tau,(-1)^jq)$ where $\Z/2\Z\cong\langle\tau\rangle=\Gal(\Q(\eta_0^i)/\Q(\psi^{(i)}_{q+1}))$; 
whereas the cyclic algebra attached to the character $\psi_{q+1}^{(i)}$ is given by $(\Q(\e^i),\tau,(-1)^i)$, see \cite[Lemma 6]{Janusz}.
We have the following isomorphisms.

\begin{lemma}\label{lemma:divisionSL}
    Let $G=\SL_2(q)$, where $q$ is a power of an odd prime $p$. Let $T$ be the irreducible $\Q$-representation of $G$ that affords the character $\Omega(\chi)$, where $\chi \in \irr(G)$. Let $A = \{ T(x) : x \in \Q G \}$ be the simple $\Q$-algebra that is isomorphic to $M_n(D)$. Then we have the following.
    \begin{enumerate}
        \item If $\chi = \psi_{q+1}^{(i)}$ with $i$ odd, then $\Q(\chi)=\Q(\zeta_{d}+\zeta_d^{-1})$, where $\Q(\epsilon^i)=\Q(\zeta_d)$ and $\epsilon \in \mathbb{C}$ is a primitive $q-1$ root of unity. Furthermore, in this case, we have $n=\frac{q+1}{2}$ and
        $$D \cong (\zeta_{d}-\zeta_d^{-1}, -1)_{\Q(\zeta_{d}+\zeta_d^{-1})}.$$

        \item If $\chi = \psi_{q-1}^{(j)}$ with $j$ odd, then $\Q(\chi)=\Q(\zeta_{d}+\zeta_d^{-1})$, where $\Q(\eta_0^j)=\Q(\zeta_d)$ and $\eta_0 \in \mathbb{C}$ is a primitive $q+1$ root of unity. Furthermore, in this case, we have $n=\frac{q-1}{2}$ and
        $$D \cong (\zeta_{d}-\zeta_d^{-1}, -q)_{\Q(\zeta_{d}+\zeta_d^{-1})}.$$

        \item If $\chi=\psi'_-,\,\psi''_-$ and $q\equiv 1\pmod 4$, then $\Q(\chi)=\Q(\sqrt q)$. The following cases arise for the division algebras;
        \begin{itemize}
            \item $q=p^{2m}$ and $p\equiv 3\pmod4$, $D\cong (-1,-p)_\Q$,
            \item $q=p^{2m}$ and $p\equiv 1\pmod 4$, $D\cong (-\ell,-p)_\Q$ where $\ell$ is a prime satisfying $\ell\equiv 3\pmod 4$ and the Legendre symbol $\left(\dfrac{\ell}{p}\right)$ takes value $-1$,
            \item $q=p^{2m+1}$, then $p\equiv 1\pmod{4}$ and $D\cong (-1,-1)_{\Q(\sqrt q)}$
        \end{itemize}
    \end{enumerate}
\end{lemma}
\begin{proof}
    First, we prove the first equality; the second follows similarly.
    The algebra $(\Q(\eta_0^{j}),\tau,(-1)^jq)$ is generated by
    \begin{align*}
        1,\eta_0^j,u_{\tau}, u_{\tau}\eta_0^j
    \end{align*}
    where
    $u_\tau^2=(-1)^iq, \text{ and }u_{\tau}\eta_0^j=\eta_0^{-j}u_{\tau}.$
    Let $a_j=(\eta_0^{j}-\eta_0^{-j})^2$. Then $\tau\left((\eta_0^{j}-\eta_0^{-j})^2\right)=(\eta_0^{j}-\eta_0^{-j})^2$ and hence $a_j\in \Q(\eta_0^{j}+\eta_0^{-j})$. Assume $b_j=(-1)^jq$. 
    Define 
    \begin{align*}
        \left(a_j,b_j\right)_{\Q}\longrightarrow (\Q(\eta_0^{j}),\tau,(-1)^jq) \text{ by }x_j\mapsto \eta_0^{j}-\eta_0^{-j}, y_j\longrightarrow (-1)^j q,
    \end{align*}
    where $\langle x_j,y_j:x_j^2=a_j,y_j^2=b_j,x_jy_j=-y_jx_j\rangle$. 
    Then the result follows from the usual check.

% Next, we determine the division algebra associated with the cyclic algebra $(\Q(\zeta_p),\tau,-1)$ where $\langle \tau\rangle =\Gal(\Q(\zeta_p)/\Q)$.
% The above cyclic algebra splits iff $-1\in \operatorname{N}_{\Q(\zeta_p)/\Q}(\Q(\zeta_p)^\times)$ iff $p\equiv 1\pmod{4}$.
% Thus, it remains to determine the algebra when $p\equiv 3 \pmod 4$.
% We claim that in this case 
% \begin{align*}
%     (\Q(\zeta_p),\tau,-1)\cong \M_{(p-1)/2}((-1,-p)_\Q).
% \end{align*}
% First note that
% \begin{align*}
%     (\Q(\zeta_p),\tau,-1)=\langle K, u:u^{p-1}=-1,ux=\tau(x)u\text{ for all }x\in \Q(\zeta_p)\rangle.
% \end{align*}
% Let $D=(-1,-p)_\Q$
% Consider the right $D$-module $V\colon =D^{(p-1)/2}$, so that $\operatorname{End}_D(V)=\M_{(p-1)/2}(D)$.
% Let $D=\langle i, j: i^2=-1,\,j^2=-p,\,ij=-ji\rangle$, and define
% \begin{align*}
%     \Phi:(\Q(\zeta_p),\tau,-1)\longrightarrow \M_{(p-1)/2}((-1,-p)_\Q)
% \end{align*}
To prove $(3)$, we use Hilbert symbols to compute the rest of the quaternion algebras; the notations are standard, see \cite{Serre1973}.
Recall that for $a,b\in\Q^\times$ one has $\prod\limits_{v}(a,b)_v=1$.
We are given $q\equiv 1\pmod{4}$; the first case we consider is when $q$ is a square. 
Two cases may arise: $p\equiv 1\pmod{4}$ and $p\equiv 3\pmod{4}$.
We determine $a$ and $b$ such that the quaternion algebra $(a,b)_\Q$ is non-split, whose Hasse invariant is nonzero at $\infty$ and $p$.
Since the Hasse invariant is positive at $\infty$, one immediately gets that $a,b<0$.

Let $p\equiv 3\pmod 4$. Then take $a=-1$ and $b=-p$. 
If $a=p^\alpha u$ and $b=p^\beta v$, one has Using the formula
\begin{align*}
    (a,b)_p=\begin{cases}
        (-1)^{\alpha\beta \epsilon(p)}\left(\dfrac{u}{p}\right)^\beta\left(\dfrac{v}{p}\right)^\alpha&\text{if }p\neq 2\\
        (-1)^{\epsilon(u)\epsilon(v)+\alpha w(v)+\beta w(u)}&\text{if }p=2
    \end{cases},
\end{align*}
where $\epsilon(u)=\dfrac{u-1}{2}\pmod{2}$, $w(u)=\dfrac{u^2-1}{2}\pmod2$ and teh Legendre symbol for a rational integer $x$ is defined to be $\left(\dfrac{x}{p}\right)\colon=x^{\frac{p-1}{2}}\pmod{\mathbb F_p^\times}$. Then we get that
\begin{align*}
    (-1,-p)_v=\begin{cases}
        -1&\text{when }v=p,\,\infty\\
        1&\text{otherwise}
    \end{cases}.
\end{align*}
Next let $p\equiv 1\pmod{4}$. 
Choose a prime integer $\ell$ such that $\ell\equiv 3\pmod{4}$ and $\left(\dfrac{\ell}{p}\right)=-1$; such a prime exists since $(p-1)/2$ elements in $\mathbb F_p^\times$ has Legendre symbol $-1$ and the series $\{4pn+r\}_{n}$ has infinitely many primes.
Then set $a=-p$ and $b=-\ell$, which gives
\begin{align*}
    (-p,-\ell)_v=\begin{cases}
        -1&\text{when }v=p,\,\infty\\
        1&\text{otherwise}
    \end{cases}.
\end{align*}

When $q\equiv1\pmod{4}$ and $q$ is not a square, we must have that $p\equiv 1\pmod{4}$. 
In this case, we need to construct $a, b$ such that the Hasse-invariant of $(a,b)_{\Q(\sqrt p)}$ is $1/2$ at each real place of $\Q(\sqrt p)$ and $0$ elsewhere.
We choose $a=-1$ and $b=-1$.
Then $a,\,b<0$ and $|-1|_\mathfrak p=1$ for all finite prime $\mathfrak p$ of $\Q(\sqrt p)$ since $-1$ is a unit in $\Q(\sqrt p)$ itself. 
Hence, the result follows.
\end{proof}

With the above setup of division rings, we now proceed to the proof of Theorem~\ref{thm:wedderburn SL(odd)}.

\begin{proof}[Proof of Theorem~\ref{thm:wedderburn SL(odd)}]
Let $G=\SL_2(q)$, where $q$ is a power of an odd prime $p$, and let $\chi \in \irr(G)$. Suppose $\rho$ is an irreducible $\mathbb{Q}$-representation of $G$ affording the character $\Omega(\chi)$. Denote by $A_\mathbb{Q}(\chi)$ the simple component in the Wedderburn decomposition of $\mathbb{Q}G$ corresponding to $\rho$, so that $A_\mathbb{Q}(\chi) \cong M_n(D)$ for some $n \in \mathbb{N}$ and a division algebra $D$. By Lemma~\ref{lemma:schurindexSL}, we have $m_\mathbb{Q}(\chi) \in \{1, 2\}$. Moreover, Lemma~\ref{Reiner} shows that $[D:Z(D)] = m_\mathbb{Q}(\chi)^2$ and $Z(D) = \mathbb{Q}(\chi)$. Hence, $D = Z(D) = \mathbb{Q}(\chi)$. Next, consider
			\[
			\rho \cong \bigoplus_{i=1}^l \rho_i,
			\]
			where $l = [\mathbb{Q}(\chi): \mathbb{Q}]$ and each $\rho_i$ is an irreducible complex representation of $G$ affording the character $\chi^{\sigma_i}$ for some $\sigma_i \in \mathrm{Gal}(\mathbb{Q}(\chi)/\mathbb{Q})$. Since $m_\mathbb{Q}(\chi) = 1$, it follows from \cite[Theorem 3.3.1]{JR} that $n = \chi(1)$. The rest of the proof proceeds in the following two cases.

            \begin{enumerate}
                \item {\bf Case ($q\equiv3\pmod{4}$).} In this case, we have
\[
			\cd(G) = \left \{1, \frac{q-1}{2}, \frac{q+1}{2}, q-1, q, q+1 \right \}.
			\]
    Note that there is only trivial representation, which is of degree one complex representation. Therefore, the simple component in the Wedderburn decomposition of $\mathbb{Q}G$ corresponding to the irreducible $\mathbb{Q}$-representations of $G$ affording the character $\Omega(\chi)$ with $\chi(1)=1$ is isomorphic to $\Q$. Furthermore, we have
     $$\Q(\psi_-')=\Q(\psi_-'')=\Q(\omega)=\Q(\sqrt{-q})$$
    (see Table~\ref{table-sl(q)}). Hence, the simple component in the Wedderburn decomposition of $\mathbb{Q}G$ corresponding to the irreducible $\mathbb{Q}$-representations of $G$ affording the character $\Omega(\chi)$ with $\chi(1)=\frac{q-1}{2}$ is isomorphic to $\M_{\frac{q-1}{2}}(\Q(\sqrt{-q}))$. Similarly, we have
     $$\Q(\psi_+')=\Q(\psi_+'')=\Q(\omega)=\Q(\sqrt{-q})$$
    (see Table~\ref{table-sl(q)}). Hence, the simple component in the Wedderburn decomposition of $\mathbb{Q}G$ corresponding to the irreducible $\mathbb{Q}$-representations of $G$ affording the character $\Omega(\chi)$ with $\chi(1)=\frac{q+1}{2}$ is isomorphic to $\M_{\frac{q+1}{2}}(\Q(\sqrt{-q}))$.
    
    Next, let $A=\left\{\, 2 < d \leq \tfrac{q+1}{2} \, : \, d \mid \tfrac{q+1}{2} \,\right\}$. By Lemma~\ref{lemma:RationRepsSL2}, for each $d \in A$, there is exactly one irreducible $\Q$-representation of $G$ that affords the character $\Omega(\chi)$ with $\chi(1)=q-1$ and $\Q(\chi)=\Q(\zeta_{d}+\zeta_d^{-1})$. Additionally, let $A'=\left\{\, 2 < m \leq q+1 \, : \, m \mid (q+1), \ m \nmid \tfrac{q+1}{2} \, \right\}$. By Lemma~\ref{lemma:RationRepsSL2}, for each $d \in A'$, there is exactly one irreducible $\Q$-representation of $G$ that affords the character $\Omega(\chi)$ with $\chi(1)=q-1$ and $\Q(\chi)=\Q(\zeta_{d}+\zeta_d^{-1})$. Moreover, these are all irreducible $\Q$-representation of $G$ that affords the character $\Omega(\chi)$ with $\chi(1)=q-1$. Therefore, Lemmas~\ref{lemma:schurindexSL} and~\ref{lemma:divisionSL} imply that the direct sum of the simple components in the Wedderburn decomposition of $\mathbb{Q}G$ corresponding to all inequivalent irreducible $\mathbb{Q}$-representations of $G$ affording the character $\Omega(\chi)$ with $\chi(1)=q-1$ is isomorphic to 
    $$\bigoplus_{d \in A'} \M_{\frac{q-1}{2}}(D_d) \bigoplus_{d \in A} \M_{q-1}(\Q(\zeta_{d}+\zeta_d^{-1})),$$
    where $D_d \cong (\zeta_{d}-\zeta_d^{-1}, -q)_{\Q(\zeta_{d}+\zeta_d^{-1})}$.
    
    Analogous to the similar argument as above, by Lemmas~\ref{lemma:schurindexSL}, \ref{lemma:divisionSL}, and~\ref{lemma:RationRepsSL2}, the direct sum of the simple components in the Wedderburn decomposition of $\mathbb{Q}G$ corresponding to all inequivalent irreducible $\mathbb{Q}$-representations of $G$ affording the character $\Omega(\chi)$ with $\chi(1)=q+1$ is isomorphic to 
    $$\bigoplus_{d \in B'} \M_{\frac{q+1}{2}}(D_d) \bigoplus_{d \in B} \M_{q+1}(\Q(\zeta_{d}+\zeta_d^{-1})),$$
    where $B=\left\{\, 2 < m \leq \tfrac{q-1}{2} \, : \, m \mid \tfrac{q-1}{2} \, \right\}$, $B'=\left\{\, 2 < m \leq q-1 \, : \, m \mid (q-1), \ m \nmid \tfrac{q-1}{2}  \, \right\}$, and the division ring $D'_d \cong  (\zeta_{d}-\zeta_d^{-1}, -1)_{\Q(\zeta_{d}+\zeta_d^{-1})}$.

    Finally, note that that there is only one irreducible complex representation of $G$ of degree $q$ that affords the character $\psi_q$. Moreover $\Q(\psi_q)=\Q$. Therefore, the simple component in the Wedderburn decomposition of $\mathbb{Q}G$ corresponding to the irreducible $\mathbb{Q}$-representations of $G$ affording the character $\Omega(\psi_q)$ is isomorphic to $\M_q(\Q)$. 

    Hence, by collecting the simple components corresponding to all inequivalent irreducible $\mathbb{Q}$-representations of $G$ affording the character $\Omega(\chi)$ with $\chi \in \irr(G)$, the result follows.

\item {\bf Case ($q\equiv1\pmod{4}$).} In this case, we have
\[
\cd(G) = \left \{1, \frac{q-1}{2}, \frac{q+1}{2}, q-1, q, q+1 \right \}.
\]

For the subcase ($q$ is not a square), we have
$$
\Q(\psi_+')=\Q(\psi_+'')=\Q(\omega)=\Q(\sqrt{q})
$$
(see Table~\ref{table-sl(q)}). By Lemma~\ref{lemma:schurindexSL}, we have $m_{\Q}(\psi_+')= m_{\Q}(\psi_+'')=1$. Therefore, it follows by Lemma~\ref{lemma:RationRepsSL2} that the simple component in the Wedderburn decomposition of $\Q G$ corresponding to the irreducible $\Q$-representations of $G$ affording the character $\Omega(\chi)$ with $\chi(1)=\frac{q+1}{2}$ is isomorphic to $\M_{\frac{q+1}{2}}(\Q(\sqrt{q}))$.

For the subcase ($q$ is a square), we have
$$
\Q(\psi_+')=\Q(\psi_+'')=\Q(\omega)=\Q
$$
(see Table~\ref{table-sl(q)}). By Lemma~\ref{lemma:schurindexSL}, we have $m_{\Q}(\psi_+')= m_{\Q}(\psi_+'')=1$. Consequently, by Lemma~\ref{lemma:RationRepsSL2}, the simple components in the Wedderburn decomposition of $\Q G$ corresponding to the irreducible $\Q$-representations of $G$ affording the character $\Omega(\chi)$ with $\chi(1)=\frac{q+1}{2}$ are isomorphic to $2\M_{\frac{q+1}{2}}(\Q)$.

Next, for the subcase ($q$ is not a square), we have
$$
\Q(\psi_-')=\Q(\psi_-'')=\Q(\omega)=\Q(\sqrt{q})
$$
(see Table~\ref{table-sl(q)}). By Lemma~\ref{lemma:schurindexSL}, we have $m_{\Q}(\psi_-')= m_{\Q}(\psi_-'')=2$. Therefore, it follows by Lemma~\ref{lemma:divisionSL} that the simple component in the Wedderburn decomposition of $\Q G$ corresponding to the irreducible $\Q$-representations of $G$ affording the character $\Omega(\chi)$ with $\chi(1)=\frac{q-1}{2}$ is isomorphic to $\M_{\frac{q-1}{4}}(\Q(D')$, where $D'\cong (-1,-1)_{\Q(\sqrt q)}$.

For the subcase ($q$ is a square), we have
$$
\Q(\psi_-')=\Q(\psi_-'')=\Q(\omega)=\Q
$$
(see Table~\ref{table-sl(q)}). By Lemma~\ref{lemma:schurindexSL}, we have $m_{\Q}(\psi_-')= m_{\Q}(\psi_-'')=1$. Let the division algebra
 $$D''\cong \begin{cases}
(-1,-p)_\Q &  \text{ for } \, \, \, \, q=p^{2m}\, \, \, \, \text{and} \, \, \, \, p\equiv 1\pmod 4, \\
 (-\ell,-p)_\Q & \text{ for } \, \, \, \, q=p^{2m}\, \, \, \, \text{and} \, \, \, \, p\equiv 3\pmod 4,
\end{cases}$$
where $\ell$ is a prime satisfying $\ell\equiv 3\pmod 4$ and the Legendre symbol $\left(\dfrac{\ell}{p}\right)$ takes value $-1$. Then by Lemma~\ref{lemma:divisionSL}, the simple components in the Wedderburn decomposition of $\Q G$ corresponding to the irreducible $\Q$-representations of $G$ affording the character $\Omega(\chi)$ with $\chi(1)=\frac{q-1}{2}$ are isomorphic to $2\M_{\frac{q-1}{4}}(D'')$.

The remaining part of the proof proceeds by arguments analogous to those in the previous case. Finally, in both subcases when $q$ is not a square and when $q$ is a square, by collecting the simple components corresponding to all inequivalent irreducible $\mathbb{Q}$-representations of $G$ affording the character $\Omega(\chi)$ with $\chi \in \irr(G)$, we obtain the respective result in each subcase.
            \end{enumerate}
	This completes the proof of Theorem~\ref{thm:wedderburn SL(odd)}.
    
\end{proof}

\section{Declarations}
\subsection{Ethical Approval} Not applicable.
\subsection{Consent for publication} Both authors agree that the manuscript may be published if accepted.
\subsection{Availability of data and material} Not applicable.
\subsection{Competing interest} We declare not to have found any competing interest regarding this article.
\subsection{Funding} No funding has been acquired for this research. 
\subsection{Authors' contributions} Both authors participated in all aspects of the contribution equally.
% \subsection{Acknowledgments} 
\printbibliography

@book {GeMa20,
    AUTHOR = {Geck, Meinolf and Malle, Gunter},
     TITLE = {The character theory of finite groups of {L}ie type},
    SERIES = {Cambridge Studies in Advanced Mathematics},
    VOLUME = {187},
      NOTE = {A guided tour},
 PUBLISHER = {Cambridge University Press, Cambridge},
      YEAR = {2020},
     PAGES = {ix+394},
      ISBN = {978-1-108-48962-1},
   MRCLASS = {20C33 (20D06 20G05)},
  MRNUMBER = {4211779},
MRREVIEWER = {Donald\ L.\ White},
}

@article {Janusz,
    AUTHOR = {Janusz, G. J.},
     TITLE = {Simple components of {$Q[{\rm SL}(2,\,q)]$}},
   JOURNAL = {Comm. Algebra},
  FJOURNAL = {Communications in Algebra},
    VOLUME = {1},
      YEAR = {1974},
     PAGES = {1--22},
      ISSN = {0092-7872,1532-4125},
   MRCLASS = {20C05},
  MRNUMBER = {344323},
MRREVIEWER = {Donald\ S.\ Passman},
       DOI = {10.1080/00927877408548606},
       URL = {https://doi.org/10.1080/00927877408548606},
}

@article {Reiner,
    AUTHOR = {Reiner, Irving},
     TITLE = {The {S}chur index in the theory of group representations},
   JOURNAL = {Michigan Math. J.},
  FJOURNAL = {Michigan Mathematical Journal},
    VOLUME = {8},
      YEAR = {1961},
     PAGES = {39--47},
      ISSN = {0026-2285,1945-2365},
   MRCLASS = {20.80},
  MRNUMBER = {122892},
MRREVIEWER = {T.\ Nakayama},
       URL = {http://projecteuclid.org/euclid.mmj/1028998513},
}

@book {JR,
    AUTHOR = {Jespers, Eric and del R\'io, \'Angel},
     TITLE = {Group ring groups. {V}ol. 1. {O}rders and generic
              constructions of units},
    SERIES = {De Gruyter Graduate},
 PUBLISHER = {De Gruyter, Berlin},
      YEAR = {2016},
     PAGES = {xii+447},
      ISBN = {978-3-11-037278-6; 978-3-11-038617-2},
   MRCLASS = {16-02 (16H10 16S34 16U60 19B28 20C05)},
  MRNUMBER = {3618092},
MRREVIEWER = {Stanley\ Orlando\ Juriaans},
}

@article {Shahabi,
    AUTHOR = {Shahabi Shojaei, M. A.},
     TITLE = {Schur indices of irreducible characters of {${\rm
              SL}(2,\,q)$}},
   JOURNAL = {Arch. Math. (Basel)},
  FJOURNAL = {Archiv der Mathematik},
    VOLUME = {40},
      YEAR = {1983},
    NUMBER = {3},
     PAGES = {221--231},
      ISSN = {0003-889X,1420-8938},
   MRCLASS = {20C15 (20G40)},
  MRNUMBER = {701268},
MRREVIEWER = {David\ B.\ Surowski},
       DOI = {10.1007/BF01192774},
       URL = {https://doi.org/10.1007/BF01192774},
}

@book {Yam,
    AUTHOR = {Yamada, Toshihiko},
     TITLE = {The {S}chur subgroup of the {B}rauer group},
    SERIES = {Lecture Notes in Mathematics},
    VOLUME = {Vol. 397},
 PUBLISHER = {Springer-Verlag, Berlin-New York},
      YEAR = {1974},
     PAGES = {v+159},
   MRCLASS = {20C05 (12A60)},
  MRNUMBER = {347957},
MRREVIEWER = {P.\ Fong},
}

@article {Herman,
    AUTHOR = {Herman, Allen},
     TITLE = {On the automorphism groups of rational group algebras of
              metacyclic groups},
   JOURNAL = {Comm. Algebra},
  FJOURNAL = {Communications in Algebra},
    VOLUME = {25},
      YEAR = {1997},
    NUMBER = {7},
     PAGES = {2085--2097},
      ISSN = {0092-7872,1532-4125},
   MRCLASS = {16S34 (16W20)},
  MRNUMBER = {1451679},
MRREVIEWER = {E.\ Jespers},
       DOI = {10.1080/00927879708825973},
       URL = {https://doi.org/10.1080/00927879708825973},
}

@article {Jes-Rio,
    AUTHOR = {Jespers, Eric and del R\'io, Angel},
     TITLE = {A structure theorem for the unit group of the integral group
              ring of some finite groups},
   JOURNAL = {J. Reine Angew. Math.},
  FJOURNAL = {Journal f\"ur die Reine und Angewandte Mathematik. [Crelle's
              Journal]},
    VOLUME = {521},
      YEAR = {2000},
     PAGES = {99--117},
      ISSN = {0075-4102,1435-5345},
   MRCLASS = {16U60 (16S34 20C10)},
  MRNUMBER = {1752297},
MRREVIEWER = {Alexander\ Zimmermann},
       DOI = {10.1515/crll.2000.032},
       URL = {https://doi.org/10.1515/crll.2000.032},
}

@article {Rit-Seh,
    AUTHOR = {Ritter, J\"urgen and Sehgal, Sudarshan K.},
     TITLE = {Construction of units in integral group rings of finite
              nilpotent groups},
   JOURNAL = {Trans. Amer. Math. Soc.},
  FJOURNAL = {Transactions of the American Mathematical Society},
    VOLUME = {324},
      YEAR = {1991},
    NUMBER = {2},
     PAGES = {603--621},
      ISSN = {0002-9947,1088-6850},
   MRCLASS = {20C05 (16S34 16U60)},
  MRNUMBER = {987166},
MRREVIEWER = {Edgar\ G.\ Goodaire},
       DOI = {10.2307/2001734},
       URL = {https://doi.org/10.2307/2001734},
}

@article {PW,
    AUTHOR = {Perlis, Sam and Walker, Gordon L.},
     TITLE = {Abelian group algebras of finite order},
   JOURNAL = {Trans. Amer. Math. Soc.},
  FJOURNAL = {Transactions of the American Mathematical Society},
    VOLUME = {68},
      YEAR = {1950},
     PAGES = {420--426},
      ISSN = {0002-9947,1088-6850},
   MRCLASS = {09.1X},
  MRNUMBER = {34758},
MRREVIEWER = {R.\ C.\ Lyndon},
       DOI = {10.2307/1990406},
       URL = {https://doi.org/10.2307/1990406},
}

@article {Ram1,
    AUTHOR = {Choudhary, Ram Karan and Prajapati, Sunil Kumar},
     TITLE = {Rational representations and rational group algebra of {VZ}
              {$p$}-groups},
   JOURNAL = {J. Aust. Math. Soc.},
  FJOURNAL = {Journal of the Australian Mathematical Society},
    VOLUME = {118},
      YEAR = {2025},
    NUMBER = {1},
     PAGES = {1--30},
      ISSN = {1446-7887,1446-8107},
   MRCLASS = {20D15 (16S34 20C15)},
  MRNUMBER = {4851707},
MRREVIEWER = {Jeffrey\ M.\ Riedl},
       DOI = {10.1017/S1446788724000132},
       URL = {https://doi.org/10.1017/S1446788724000132},
}

@article {Ram2,
    AUTHOR = {Choudhary, Ram Karan and Prajapati, Sunil Kumar},
     TITLE = {Rational group algebras of {C}amina {$p$}-groups},
   JOURNAL = {Arch. Math. (Basel)},
  FJOURNAL = {Archiv der Mathematik},
    VOLUME = {125},
      YEAR = {2025},
    NUMBER = {3},
     PAGES = {247--258},
      ISSN = {0003-889X,1420-8938},
   MRCLASS = {20C05 (20C15 20D15)},
  MRNUMBER = {4955643},
       DOI = {10.1007/s00013-025-02142-w},
       URL = {https://doi.org/10.1007/s00013-025-02142-w},
}

@article {Ram4,
    AUTHOR = {Choudhary, Ram Karan and Prajapati, Sunil Kumar},
     TITLE = {A combinatorial formula for the {W}edderburn
              decomposition of rational group algebras and the
              rational representations of ordinary metacyclic
              {$p$}-groups},
   JOURNAL = {Algebr. Represent. Theory},
  FJOURNAL = {Algebras and Representation Theory},
    VOLUME = {29},
      YEAR = {2026},
    NUMBER = {1},
     PAGES = {85--109},
      ISSN = {1386-923X,1572-9079},
   MRCLASS = {99-06 (20C15 20D15)},
  MRNUMBER = {5020891},
       DOI = {10.1007/s10468-025-10371-4},
       URL = {https://doi.org/10.1007/s10468-025-10371-4},
}

@article {Ram5,
    AUTHOR = {Choudhary, Ram Karan and Prajapati, Sunil Kumar},
     TITLE = {A combinatorial technique for the Wedderburn decomposition of rational group algebras of nested GVZ {$p$}-groups},
   JOURNAL = {J. Algebraic Combin.},
  FJOURNAL = {Joural of Algebraic Combinatorics},
      YEAR = {2026},
      VOLUME = {63},
      NUMBER = {39},
      PAGES = {(to appear)},
       URL = {https://doi.org/10.1007/s10801-026-01527-6},
       DOI = {doi.org/10.1007/s10801-026-01527-6},
}

@article {Ram6,
    AUTHOR = {Choudhary, Ram Karan and Prajapati, Sunil Kumar},
     TITLE = {On matrix representations of groups of order {$p^5$} over {$\mathbb{Q}$}},
   JOURNAL = {J. Group Theory},
  FJOURNAL = {Joural of Group Theory},
      YEAR = {2026},
      PAGES = {(to appear)},
       
}

@article {Ram3,
    AUTHOR = {Choudhary, Ram Karan and Prajapati, Sunil Kumar},
     TITLE = {A combinatorial formula for the Wedderburn decomposition of rational group algebras of split metacyclic {$p$}-groups},
   JOURNAL = {J. Algebra Appl.},
  FJOURNAL = {Joural of Algebra and its Applications},
      YEAR = {2026},
      PAGES = {(to appear)},
       URL = {https://doi.org/10.1142/S0219498826500684},
       DOI = {doi.org/10.1142/S0219498826500684},
}

@manual{Gap,
    key          = "GAP",
    organization = "\href{https://www.gap-system.org}{https://www.gap-system.org}",
    title        = "{The GAP~Group, GAP -- Groups, Algorithms, and Programming,
                    Version 4.13.0}",
    year         = 2024,
    url          = "\href{https://www.gap-system.org}{https://www.gap-system.org}"
    }

@book {I,
    AUTHOR = {Isaacs, I. Martin},
     TITLE = {Character theory of finite groups},
      NOTE = {Corrected reprint of the 1976 original [Academic Press, New
              York; MR0460423 (57 \#417)]},
 PUBLISHER = {Dover Publications, Inc., New York},
      YEAR = {1994},
     PAGES = {xii+303},
      ISBN = {0-486-68014-2},
   MRCLASS = {20C15},
  MRNUMBER = {1280461},
}

@book {Serre1973,
    AUTHOR = {Serre, J.-P.},
     TITLE = {A course in arithmetic},
    SERIES = {Graduate Texts in Mathematics},
    VOLUME = {No. 7},
      NOTE = {Translated from the French},
 PUBLISHER = {Springer-Verlag, New York-Heidelberg},
      YEAR = {1973},
     PAGES = {viii+115},
   MRCLASS = {12-02 (10CXX 10DXX)},
  MRNUMBER = {344216},
}

@article {Shoda1,
    AUTHOR = {Olteanu, Gabriela},
     TITLE = {Computing the {W}edderburn decomposition of group algebras by
              the {B}rauer-{W}itt theorem},
   JOURNAL = {Math. Comp.},
  FJOURNAL = {Mathematics of Computation},
    VOLUME = {76},
      YEAR = {2007},
    NUMBER = {258},
     PAGES = {1073--1087},
      ISSN = {0025-5718,1088-6842},
   MRCLASS = {16S34 (20C15)},
  MRNUMBER = {2291851},
MRREVIEWER = {E.\ Jespers},
       DOI = {10.1090/S0025-5718-07-01957-6},
       URL = {https://doi.org/10.1090/S0025-5718-07-01957-6},
}

@article {Shoda2,
    AUTHOR = {Bakshi, Gurmeet K. and Maheshwary, Sugandha},
     TITLE = {The rational group algebra of a normally monomial group},
   JOURNAL = {J. Pure Appl. Algebra},
  FJOURNAL = {Journal of Pure and Applied Algebra},
    VOLUME = {218},
      YEAR = {2014},
    NUMBER = {9},
     PAGES = {1583--1593},
      ISSN = {0022-4049,1873-1376},
   MRCLASS = {16S34 (20C05)},
  MRNUMBER = {3188857},
MRREVIEWER = {Adalbert\ Bovdi},
       DOI = {10.1016/j.jpaa.2013.12.010},
       URL = {https://doi.org/10.1016/j.jpaa.2013.12.010},
}

@article {Shoda3,
    AUTHOR = {Bakshi, Gurmeet K. and Garg, Jyoti and Olteanu, Gabriela},
     TITLE = {Rational group algebras of generalized strongly monomial
              groups: primitive idempotents and units},
   JOURNAL = {Math. Comp.},
  FJOURNAL = {Mathematics of Computation},
    VOLUME = {93},
      YEAR = {2024},
    NUMBER = {350},
     PAGES = {3027--3058},
      ISSN = {0025-5718,1088-6842},
   MRCLASS = {16K20 (16S35 16U60 20C05)},
  MRNUMBER = {4780354},
MRREVIEWER = {Gurleen\ Kaur},
       DOI = {10.1090/mcom/3937},
       URL = {https://doi.org/10.1090/mcom/3937},
}

@article {Shoda4,
    AUTHOR = {Jespers, Eric and Olteanu, Gabriela and del R\'io, \'Angel},
     TITLE = {Rational group algebras of finite groups: from idempotents to
              units of integral group rings},
   JOURNAL = {Algebr. Represent. Theory},
  FJOURNAL = {Algebras and Representation Theory},
    VOLUME = {15},
      YEAR = {2012},
    NUMBER = {2},
     PAGES = {359--377},
      ISSN = {1386-923X,1572-9079},
   MRCLASS = {16S34 (16U60 20C05)},
  MRNUMBER = {2892512},
       DOI = {10.1007/s10468-010-9244-4},
       URL = {https://doi.org/10.1007/s10468-010-9244-4},
}

@article {Shoda5,
    AUTHOR = {Garc\'ia-Bl\'azquez, \`Angel and del R\'io, \'Angel},
     TITLE = {The isomorphism problem for rational group algebras of finite
              metacyclic groups},
   JOURNAL = {J. Pure Appl. Algebra},
  FJOURNAL = {Journal of Pure and Applied Algebra},
    VOLUME = {229},
      YEAR = {2025},
    NUMBER = {6},
     PAGES = {Paper No. 107951, 36},
      ISSN = {0022-4049,1873-1376},
   MRCLASS = {16S34 (20C05)},
  MRNUMBER = {4881589},
       DOI = {10.1016/j.jpaa.2025.107951},
       URL = {https://doi.org/10.1016/j.jpaa.2025.107951},
}

@article {Shoda6,
    AUTHOR = {Nachev, Nako and Epitropov, Yordan},
     TITLE = {Isomorphism of group algebras of metacyclic groups over the
              field of rational numbers},
   JOURNAL = {J. Algebra},
  FJOURNAL = {Journal of Algebra},
    VOLUME = {680},
      YEAR = {2025},
     PAGES = {58--69},
      ISSN = {0021-8693,1090-266X},
   MRCLASS = {16S34 (16U40 20C05)},
  MRNUMBER = {4913648},
       DOI = {10.1016/j.jalgebra.2025.04.040},
       URL = {https://doi.org/10.1016/j.jalgebra.2025.04.040},
}
\vspace{2em}
\end{document}